\documentclass[11pt,oneside,reqno]{amsart}
\usepackage{amssymb}
\usepackage{amsmath}
\usepackage{graphicx}
\usepackage{mathrsfs}
\usepackage{amsfonts}
\usepackage{enumerate,amsmath,amssymb,amsthm}
\usepackage{hyperref}
\usepackage{bbm}

\hypersetup{backref,
colorlinks=true,
linkcolor=blue,
anchorcolor=blue,
citecolor=blue}

\numberwithin{equation}{section}

\newcommand{\be}{\begin{eqnarray}}
\newcommand{\ee}{\end{eqnarray}}
\newcommand{\ce}{\begin{eqnarray*}}
\newcommand{\de}{\end{eqnarray*}}
\newtheorem{theorem}{Theorem}[section]
\newtheorem{lemma}[theorem]{Lemma}
\newtheorem{claim}[theorem]{Claim}
\newtheorem{remark}[theorem]{Remark}
\newtheorem{definition}[theorem]{Definition}
\newtheorem{proposition}[theorem]{Proposition}
\newtheorem{example}[theorem]{Example}
\newtheorem{corollary}[theorem]{Corollary}
\newtheorem{assumption}{Assumption}[section]

\def \eqref#1{\hbox{(\ref{#1})}}

\def\[{{\Big[}}
\def\]{{\Big]}}
\def\<{{\langle}}
\def\>{{\rangle}}
\def\({{\Big(}}
\def\){{\Big)}}

\def\bx{{\mathbf{x}}}

\def\={&\!\!=\!\!&}
\def\bt{\begin{theorem}}
\def\et{\end{theorem}}
\def\bt{\begin{claim}}
\def\et{\end{claim}}
\def\bl{\begin{lemma}}
\def\el{\end{lemma}}
\def\br{\begin{remark}}
\def\er{\end{remark}}
\def\bas{\begin{assumption}}
\def\eas{\end{assumption}}
\def\bd{\begin{definition}}
\def\ed{\end{definition}}
\def\bp{\begin{proposition}}
\def\ep{\end{proposition}}
\def\bc{\begin{corollary}}
\def\ec{\end{corollary}}
\def\bx{\begin{example}}
\def\ex{\end{example}}

\def\geq{\geqslant}
\def\leq{\leqslant}

\allowdisplaybreaks
\usepackage{color}

\title[SVI Solutions]{\bf{SVI solutions to stochastic nonlinear diffusion equations on general measure spaces}}

\author{Benjamin Gess}
\address{Benjamin Gess, Faculty of Mathematics, University of Bielefeld, D-33615 Bielefeld, Germany;
Max Planck Institute for Mathematics in the Sciences, 04103 Leipzig, Germany.}
\thanks{Benjamin Gess is supported by the Deutsche Forschungsgemeinschaft (DFG, German Research Foundation) through CRC 1283, and by the Max Planck Society through the Research Group ``Stochastic Analysis in the Sciences".}
\email{bgess@math.uni-bielefeld.de}

\author{Michael R\"{o}ckner}
\address{Michael R\"{o}ckner, Faculty of Mathematics, University of Bielefeld, D-33615 Bielefeld, Germany;
Academy of Mathematics and Systems Science, Chinese Academy of Sciences, Beijing 100190, China.}
\thanks{Michael R\"{o}ckner is supported by the DFG through CRC 1283.}
\email{roeckner@math.uni-bielefeld.de}

\author{Weina Wu}
\address{Weina Wu, School of Economics, Nanjing University of Finance and Economics, Nanjing, Jiangsu 210023, China; Faculty of Mathematics, University of Bielefeld, D-33615 Bielefeld, Germany.}
\thanks{Weina Wu is supported by the National Natural Science Foundation of China (NSFC) (No.11901285), China Scholarship Council (CSC) (No.202008320239), DFG through CRC 1283}
\email{wuweinaforever@163.com}

\date{\today}
\begin{document}

\maketitle

\begin{abstract}
We establish a framework for the existence and uniqueness of solutions to stochastic nonlinear (possibly multi-valued) diffusion equations driven by multiplicative noise, with the drift operator $L$ being the generator of a transient Dirichlet form on a finite measure space $(E,\mathcal{B},\mu)$ and the initial value in $\mathcal{F}_e^*$, which is the dual space of an extended transient Dirichlet space. $L$ and $\mathcal{F}_e^*$ replace the Laplace operator $\Delta$ and $H^{-1}$, respectively, in the classical case. This framework includes stochastic fast diffusion equations, stochastic fractional fast diffusion equations, the Zhang model, and apply to cases with $E$ being a manifold, a fractal or a graph. In addition, our results apply to operators $-f(-L)$, where $f$ is a Bernstein function, e.g. $f(\lambda)=\lambda^\alpha$ or $f(\lambda)=(\lambda+1)^\alpha-1$, $0<\alpha<1$.

\end{abstract}

\tableofcontents

\section{Introduction}\label{Introduction}

The concept of variational inequalities (VIs) was first introduced by Stampacchia \cite{Stampacchia} and Lions-Stampacchia \cite{LionsStampacchia} in order to study regularity problems for partial differential equations. VIs were later used as a technique to solve optimal stopping-time problems and produced numerous new results. E.g., they allow to obtain the regularity of solutions under very weak regularity assumptions for the data, provide a very convenient
characterisation of the solution from the algorithmic point of
view and they turn out to be particularly useful to treat evolutionary
cases (cf. \cite[Chapter 3]{BensoussanLions} for further details). The notion of stochastic variational inequalites (SVIs) was developed by Rascanu, Haussmann-Pardoux and Rascanu-Bensoussan in a series of works (cf. \cite{Rascanu, HaussmannPardoux, BRSVI} and the references therein). As pointed out in \cite{BRSVI}, it arouse from the following problems. Considering the evolution of a state $X_t$ of a dynamical system perturbed by noise of the form:
\begin{equation*}\label{SPE1}
dX_t=f(t,X_t)dt+g(t,X_t)dW_t,
\end{equation*}
where $X_t$ takes value in $\Bbb{R}^d$ and
$W_t$ is a $d$-dimensional Wiener process. If $g$ is not degenerate, i.e., $gg^*\geq\alpha I$, $\alpha>0$, then $X_t$ can take any value in $\Bbb{R}^d$. In other words, $X_t$ can be located in any domain $\mathcal{O}\subset\Bbb{R}^d$ with strictly positive probability. However, for certain applications, one might prefer $X_t$ to remain within a convex domain $\Bar{\mathcal{O}}\subset\Bbb{R}^d$. For instance, one might want to constrain $X_t\in[a,b]$ when $d=1$. Consequently, the following model should be considered:
\begin{equation}\label{SPE2}
\left\{ \begin{aligned}
&dX_t+\partial I_{\Bar{\mathcal{O}}}(X_t)dt\ni f(t,X_t)dt+g(t,X_t)dW_t,\\
&X_0=x_0\in \Bar{\mathcal{O}},
\end{aligned} \right.
\end{equation}
where $\partial I_{\Bar{\mathcal{O}}}$ denotes the subdifferential of $I_{\Bar{\mathcal{O}}}$, which is defined by
\begin{eqnarray*}
I_{\Bar{\mathcal{O}}}(x):=\left\{
           \begin{array}{ll}
             &\!\!\!\!\!\!0,~~~~~~~~~~~~~~~\text{if}~x\in\Bar{\mathcal{O}},\\
             &\!\!\!\!\!\!+\infty,\ \text{if}~x\in\Bbb{R}^d\setminus \Bar{\mathcal{O}}.
           \end{array}
         \right.
\end{eqnarray*}
This gives rise to the motivation for studying the following more generalized model:
\begin{equation}\label{SPE3}
\left\{ \begin{aligned}
&dX_t+\partial \varphi(X_t)dt\ni f(t,X_t)dt+g(t,X_t)dW_t,\\
&X_0=x_0\in Dom(\varphi),
\end{aligned} \right.
\end{equation}
where $\partial\varphi$ denotes the subdifferential of a proper convex lower-semicontinuous function $\varphi$, and $\partial\varphi$ is a maximal monotone function (possibly multi-valued), hence leading to the notation $``\ni"$.

To treat equations like \eqref{SPE3}, alternative concepts of strong solutions, weak solutions (or SVI solutions) and almost weak solutions to equation \eqref{SPE3} have been proposed (see e.g. \cite{Rascanu}, \cite{BRSVI}). The readers are refer to Definitions \ref{svi} and \ref{strong solution} below for the key differences between SVI solutions and strong solutions. In recent years, SVI solutions have also been developed by Barbu-Da Prato-R\"ockner, Gess-R\"ockner, Gess-T\"olle (see e.g. \cite{BDRSIAM}, \cite{GRSIAM}, \cite{GTJDE}, \cite{N} and the references therein) as an effective approach when studying some types of multi-valued stochastic partial differential equations (SPDEs) which can be written as gradient flow equations. These provide a strategy to gain the information about the structure of the limit process of the approximation equations. In addition, the limit process is allowed to have its initial value in a larger space and the assumptions on the coefficients of the SPDE can be weakened.

Before we introduce the general class of equations we are concerned with in the present paper, let us start with the following special case, i.e., stochastic generalized porous media equations (PMEs):
\begin{equation} \label{SPME1}
\left\{ \begin{aligned}
&dX_t\in \Delta\beta(X_t)dt+B(t,X_t)dW_t,\ \text{in}\ [0,T]\times \mathcal{O},\\
&X_0=x_0\in H^{-1}.
\end{aligned} \right.
\end{equation}
Here $\Delta$ is the Dirichlet Laplacian with zero boundary conditions on a bounded domain $\mathcal{O}$ of $\Bbb{R}^d$, $d\geq3$. $\beta (:=\partial\psi): \mathbb{R}\rightarrow\!2^\mathbb{R}$ is the subdifferential of $\psi:\Bbb{R}\rightarrow [0,\infty)$, where $\psi$ is a convex, lower-semicontinuous function (cf. \textbf{(H1)}-\textbf{(H3)} in Section \ref{Assumptions and main results} for the exact conditions on $\beta$ and $\psi$, see also Examples \ref{example FDE} and \ref{example Zhang} below). It is well known that $\beta$ is a maximal monotone multi-valued function (cf. Remark \ref{H1 and H3}). Furthermore, $H^{-1}$ is the dual space of $H_0^1(\mathcal{O})$, which is the usual
Sobolev space of order two in $L^2(\mathcal{O}, dx)$ ($dx$ denotes the Lebesgue measure). $B$ is a Hilbert-Schmidt operator-valued map fulfilling certain Lipschitz and growth conditions.

Equation \eqref{SPME1} can be reformulated as a gradient flow equation:
\begin{equation*} \label{SPME2}
\left\{ \begin{aligned}
&dX_t\in -\partial\varphi(X_t)dt+B(t,X_t)dW_t,\ \text{in}\ [0,T]\times \mathcal{O},\\
&X_0=x_0\in H^{-1},
\end{aligned} \right.
\end{equation*}
where $\varphi$ is a proper potential function (see \cite{GRSIAM, N}). Equation \eqref{SPE2} is a typical case where $\partial\varphi$ is in fact multi-valued. But there is also a strong motivation for studying the multi-valued case originating from physics. The exploration of \eqref{SPME1} expands upon the study of the classical PME perturbed by multiplicative noise, as follows:
\begin{equation} \label{PME}
\left\{ \begin{aligned}
&dX_t=\Delta(|X_t|^\gamma)dt+B(t,X_t)dW_t,\ \text{in}\ [0,T]\times \mathcal{O},\\
&X_0=x_0\in H^{-1}\text{~on~} \mathcal{O},
\end{aligned} \right.
\end{equation}
where $\gamma>1$. There are numerous applications of \eqref{PME} in the deterministic case, i.e., when $B:=0$, in physics. These include describing the flow of an isentropic gas through porous media, heat radiation in plasmas, the study of groundwater infiltration and so on. We refer to \cite{V, PMEBOOK} for more details on PMEs, and \cite{BDR1} for more details on stochastic PMEs. By extending $\Delta (|\cdot|^\gamma)$ to a maximal monotone operator $\Delta\beta(\cdot)$ as in \eqref{SPME1}, the framework becomes applicable to stochastic fast diffusion equations (FDEs), which is a crucial model for singular nonlinear (density-dependent) diffusive phenomena and has applications across various fields, see e.g. Example \ref{example FDE}. Analytically, the rationale for considering $\beta$ as a multi-valued function is the following. Applying the chain rule, we obtain:
\begin{eqnarray}\label{eq:1.20}
\Delta\beta(X_t)=\beta'(X_t)\Delta X_t+\beta''(X_t)|\nabla X_t|^2.
\end{eqnarray}
This shows that $\beta'(X_t)$ is the diffusion coefficient of the equation, which is dependent on the solution. This explains why $\beta$ is usually called the ``diffusivity (function)" and why it must be assumed to be increasing. If $\beta$ is strictly increasing, which corresponds to $\beta'>0$ on $\Bbb{R}$ in \eqref{eq:1.20}, then we would have local strict ellipticity, hence we would be in the non-degenerate case. However, we do not make this assumption in this paper, so the degenerate case is covered. \eqref{eq:1.20} reveals why it is important to include multi-valued diffusivities, as it allows us to cover non-continuous $\beta$ (see Example \ref{example Zhang}). This means that its generalized derivative $\beta'$ (in the sense of Schwartz distributions) would be a weighted Dirac measure $\delta_{r_0}$ at a point of discontinuity $r_0$ of $\beta$. Therefore, if we consider the time evolution $t\mapsto X_t(\xi)$ of the substance density at a point $\xi\in \mathcal{O}$ and it encounters a discontinuity point $r_0\in\Bbb{R}$ of $\beta$, the diffusion coefficient $\beta'(X_t(\xi))$ would jump to $+\infty$, indicating a ``very large" diffusion of the system at that moment. This is an interesting case of high relevance, especially in physics, and it is also the reason why the solution to \eqref{SPME1} are sometimes called singular diffusions. An important example is the Zhang model that descrides self-organized criticality (SOC), which has applications in various dynamical systems, e.g., earthquake mechanisms, forest fires, and sandpiles, etc., see Example \ref{example Zhang}.

In fact, equation \eqref{SPME1} is also of high interest if we replace the Laplace operator $\Delta$ by other operators. For example, nonlocal pseudodifferential operators, such as the fractional Laplacian $-(-\Delta)^{\frac{\alpha}{2}}$, $0<\alpha\leq2$. These are infinitesimal generators of stable L\'{e}vy processes (see e.g. Corollary 3.4.11 and Example 3.4.13 in \cite{David}), and they relate to models in physics, such as the study of hydrodynamic limits of interacting particle systems with long-range dynamics (see \cite[Appendix B]{BV} and the references therein). More generally, one considers nonlocal operators of type $-f(-\Delta)$, where $f$ is a Bernstein function (see \cite{SSV}), or even more broadly, by a self-adjoint operator $L$ satisfying certain properties. It turns out that a suitable general enough framework to cover all these operators is to assume that they are generators of transient Dirichlet forms on $L^2(\mathcal{O},dx)$. Additionally, many interesting operators on $\mathcal{O}$ are not self-adjoint on $L^2(\mathcal{O},dx)$, but on $L^2(\mathcal{O},\mu)$ for some other measure $\mu$ replacing $dx$. An example is the Friedrichs extension of the operator
$L_0=\Delta +2\frac{\nabla\rho}{\rho}\cdot\nabla$ on $L^2(\mathcal{O}, \rho^2dx)$ (\cite{RWX}), where $\rho\in H^1_0(\mathcal{O})$. Moreover, one is also interested in more general ``state spaces" than $\mathcal{O}$, for example, a smooth or even non-smooth Riemannian manifold, such as a fractal, or allow infinite-dimensional state spaces, such as the Wiener space.

To meet all these demands in generality, we need to extend the study of \eqref{SPME1} to a sufficiently wide framework as follows:
\begin{equation} \label{eq:1}
\left\{ \begin{aligned}
&dX_t\in L\beta(X_t)dt+B(t,X_t)dW_t,\ \text{in}\ [0,T]\times E,\\
&X_0=x_0\in\mathcal{F}_e^*\text{~on~} E.
\end{aligned} \right.
\end{equation}
Here $(E,\mathcal{B},\mu)$ is a standard measurable space (\cite[page 133, Definition 2.2]{P67}) with a finite measure $\mu$, $(L,D(L))$ is the generator of a symmetric strongly continuous contraction sub-Markovian semigroup $\{P_t\}_{t>0}$ on $L^2(\mu)$, which additionally is assumed to be the generator of a transient Dirichlet form $(\mathcal{E}, D(\mathcal{E}))$ (cf. Section \ref{Dirichlet spaces} and note that $D(\mathcal{E})$ replaces $H_0^1(\mathcal{O})$ from above, more precisely the homogeneous Sobolev spaces $\dot{H}^1$). Throughout the paper, for $1\leq p\leq\infty$, we write $L^p(\mu)$ for $L^p(E,\mathcal{B},\mu)$ and $|\cdot|_p$ for the norm in $L^p(\mu)$. $W$ is a cylindrical Wiener process on some separable Hilbert space $U$ defined on a probability space $(\Omega,\mathcal{F},\Bbb{P})$ with normal filtration $(\mathcal{F}_t)_{t\geq0}$. $B$ is a Hilbert-Schmidt operator-valued map fulfilling certain Lipschitz and growth conditions (cf. \textbf{(H4)} in Section \ref{Assumptions and main results}). $\mathcal{F}_e^*$ is the dual space of the transient extended Dirichlet space $\mathcal{F}_e$ of $(\mathcal{E}, D(\mathcal{E}))$ (cf. Section \ref{Dirichlet spaces} and note that $\mathcal{F}_e^*$ replaces $H^{-1}$ from above, more precisely $(\dot{H}^{1})^*$, the dual space of $\dot{H}^1$). $\beta (:=\partial\psi): \mathbb{R}\rightarrow\!2^\mathbb{R}$ is defined as in \eqref{SPME1}, and satisfies \textbf{(H1)}-\textbf{(H3)} in Section \ref{Assumptions and main results}.

A series of papers have contributed to the study of \eqref{SPME1} on a domain of $\Bbb{R}^d$, see e.g., \cite{BDPTRF, BDRAOP, BDRJMAA, BDRSIAM, BDRSIAM2, BDRCMP, BDRCMP2012, BRARCH, BRR, GCMP, GRSIAM, GRTRAN, GTJDE, N}, and \eqref{eq:1} on general measure spaces but with $\beta$ being a single-valued function, see e.g., \cite{RRW, RW, RWX} and the references therein. However, as far as we know, there are only two papers that address \eqref{eq:1} on general measure spaces with $\beta$ being a multi-valued function. The existence and uniqueness of strong solutions (see Definition \ref{strong solution}) to \eqref{eq:1} for more regular initial values $x_0\in L^2(\mu)\cap L^{2m}(\mu)\cap\mathcal{F}_e^*$, where $m\in[1,\infty)$, was established in \cite{RWX2021}, however, requiring more restrictive assumptions on the Dirichlet forms $(\mathcal{E}, D(\mathcal{E}))$ (see Remark \ref{RWX strong solution}). Furthermore, in \cite[Example 7.3]{GTJMPA} the well-posedness of ``generalized solutions" (see \cite[page 211]{book Daprato}) to \eqref{eq:1} were proved under the assumption $D(\mathcal{E})=\mathcal{F}_e$. In \cite{GTJMPA}, such solutions were called ``limit solutions" (see \cite[ Definition 4.5]{GTJMPA}.) As an improvement, in our paper, under more general assumptions we prove that these ``generalized solutions" are exactly the solutions in the sense of SVI to \eqref{eq:1} (see Definition \ref{svi}). This shows an important characteristic of SVI solutions, namely they give sense to ``generalized solutions" as a true solution to stochastic (partial) differential equations in the SVI sense.

We will employ the approach introduced in \cite{GRSIAM} (which was also used in \cite{N}). Given that we are working with a general measure space, the techniques available for Euclidean spaces in \cite{GRSIAM} cannot be directly applied to our framework. Consequently, it becomes necessary to consider the space $L^1(\mu)\cap\mathcal{F}_e^*$  (cf. Section \ref{Lower semicontinuity of varphi}), prove the lower semicontinuity of the energy potential functional (cf. Proposition \ref{cvx}), and construct appropriate convergent sequences in
$\mathcal{F}_e^*$ that satisfy certain properties (cf. Proposition \ref{construction of un}). The main difficulty is to prove the uniqueness of SVI solutions to \eqref{eq:1}, in which Proposition \ref{construction of un} plays a key role. Since it is of independent interest, and its proof is independent of the rest of the paper, we include Proposition \ref{construction of un} in Appendix \ref{Appendix}. For the proof of Proposition \ref{construction of un}, we need to use the semigroup $\{P_t\}_{t>0}$ as a mollifier, which is, therefore, assumed to be ultrabounded (cf. Remark \ref{ultraboundeded}). This is necessary to replace the use of  convolution as a mollifier in the special case $E:=\mathcal{O}$ in \cite{N}, which is also used in \cite{GRSIAM}. Finally, in Appendix \ref{Appendix} we compare the notion of SVI solutions to \eqref{eq:1} to the notion of strong solutions to \eqref{eq:1} in the sense of \cite{RWX2021} (see Definition \ref{strong solution} below).


The remaining sections of the paper are organized as follows: In Section \ref{Notations and perliminaries}, we introduce notations and provide a brief summary of some known results used in this paper. In Section \ref{Assumptions and main results}, we present our assumptions, the definition of SVI solutions to \eqref{eq:1} and state our main theorem (Theorem \ref{theorem}). Section \ref{applications} is devoted to applications of our main result. To enhance the clarity of the proofs' structure, some lemma and propositions are put in Appendix \ref{Appendix}.

\section{Preliminaries and notations}\label{Notations and perliminaries}

\vspace{2mm}
\subsection{Dirichlet spaces}\label{Dirichlet spaces}

Consider the $\Gamma$-transform $V_r (r>0)$ of a symmetric strongly continuous contraction sub-Markovian semigroup $\{P_t\}_{t>0}$ on $L^2(\mu)$:
\begin{eqnarray*}
V_rw=\frac{1}{\Gamma(\frac{r}{2})}\int_0^\infty
t^{\frac{r}{2}-1}e^{-t}P_twdt,~~w\in L^2(\mu).
\end{eqnarray*}
Define the Bessel-potential space $(F_{1,2},\|\cdot\|_{F_{1,2}})$ (\cite{Fukushima}) by
$$
F_{1,2}:=V_1(L^2(\mu)),~\text{with~norm}~\|u\|_{F_{1,2}}=|w|_2,~~\text{for}~~u=V_1w,~~ w\in L^2(\mu),
$$
consequently,
$$
V_1=(1-L)^{-\frac{1}{2}},~~F_{1,2}=D\big((1-L)^{\frac{1}{2}}\big)~~\text{and}~~\|u\|_{F_{1,2}}=|(1-L)^{\frac{1}{2}}u|_2,
$$
where $(L,D(L))$ is the generator of $\{P_t\}_{t>0}$ on $L^2(\mu)$. The dual space of $F_{1,2}$ is denoted by $F^*_{1,2}$ and $F^*_{1,2}=D((1-L)^{-\frac{1}{2}})$.

\vspace{2mm}
The Dirichlet form $(\mathcal{E}, D(\mathcal{E}))$ of $\{P_t\}_{t>0}$ on $L^2(\mu)$ associated with $(L, D(L))$ is by definition given by
\begin{eqnarray*}
   && D(\mathcal{E})=D(\sqrt{-L}), \\
   && \mathcal{E}(u,v)=\mu(\sqrt{-L}u \sqrt{-L}v),~~ \forall u,v\in D(\mathcal{E}),
\end{eqnarray*}
and accordingly we have the identification
\begin{eqnarray*}\label{F12Fe}
F_{1,2}=D(\mathcal{E}),~~~\|u\|_{F_{1,2}}^2=\mathcal{E}(u,u)+\langle u,u\rangle_2,
\end{eqnarray*}
where $\langle\cdot, \cdot\rangle_2$ denotes the inner product in $L^2(\mu)$.

Recall from \cite[page 40]{FOT} that a Dirichlet space $(\mathcal{E}, F_{1,2})$ relative to $L^2(\mu)$ is transient if there exists a bounded $\mu$-integrable function $g$ strictly positive $\mu$-a.s. on $E$ such that
\begin{eqnarray*}
\int_E|u|gd\mu\leq\sqrt{\mathcal{E}(u,u)},~~~\forall u\in F_{1,2}.
\end{eqnarray*}
If $(\mathcal{E}, F_{1,2})$ is transient, then let $\mathcal{F}_e$ be the completion of $F_{1,2}$ with respect to the norm
$$\|\cdot\|_{\mathcal{F}_e}:=\mathcal{E}(\cdot,\cdot)^{\frac{1}{2}},$$
in this case, from \cite[page 43, Theorem 1.5.3]{FOT}, we know that $F_{1,2}=\mathcal{F}_e\cap L^2(\mu)$, and $F_{1,2}$ is dense both in $L^2(\mu)$ and in $\mathcal{F}_e$. Let $\mathcal{F}^*_e$ be the dual space of $\mathcal{F}_e$ with inner product
$\langle\cdot,\cdot\rangle_{\mathcal{F}^*_e}$ and corresponding norm
$\|\cdot\|_{\mathcal{F}^*_e}$, which is induced by the Riesz map
$\mathcal{F}_e\ni u\mapsto \mathcal{E}(\cdot,u)\in \mathcal{F}^*_e$. $\mathcal{F}^*_e$ is also a Hilbert space. Since $F_{1,2}\subset\mathcal{F}_e$ continuously and densely, we have $\mathcal{F}^*_e\subset F^*_{1,2}$ also continuously and densely. For more background knowledge on Dirichlet forms, we refer to \cite{FOT, MR}.

\subsection{Some known results}\label{Some known results}

From now on, we need to assume:

\vspace{2mm}

\textbf{(T)}~~ The symmetric Dirichlet form $(\mathcal{E},D(\mathcal{E}))$ associated with $(L,D(L))$ is transient.

\vspace{2mm}
Consider the inner product
$\mathcal{E}_\nu:=\mathcal{E}+\nu \langle\cdot, \cdot \rangle_2,~ \nu\in(0,\infty)$, on
$F_{1,2}$, i.e.,
\begin{eqnarray*}
  \|v\|^2_{F_{1,2,\nu}}:=\mathcal{E}(v,v)+\nu\int|v|^2d\mu=\|v\|^2_{\mathcal{F}_e}+\nu\int|v|^2d\mu,\ \forall v\in F_{1,2},
\end{eqnarray*}
and
\begin{eqnarray*}\label{F12}
 \|l\|_{F^*_{1,2,\nu}}:=_{F^*_{1,2}}\langle l, (\nu-L)^{-1}l\rangle^{\frac{1}{2}}_{F_{1,2}}:=\sup_{\substack{v\in
F_{1,2}\\ \|v\|_{F_{1,2,\nu}}\leq1}}l(v),\ \forall l\in F_{1,2}^*,
\end{eqnarray*}
\begin{eqnarray*}
 \|l\|_{\mathcal{F}^*_e}:=\sup_{\substack{v\in
\mathcal{F}_e\\ \|v\|_{\mathcal{F}_e}\leq1}}l(v),\ \forall l\in\mathcal{F}^*_e.
\end{eqnarray*}
Recall from \cite[Propostion 2.1]{RWX2021} that
\begin{eqnarray}\label{lim1}
\lim_{\nu\rightarrow0}\|l\|_{F^*_{1,2,\nu}}=\sup_{\nu>0}\|l\|_{F^*_{1,2,\nu}}=\|l\|_{\mathcal{F}_e^*},\ \forall l\in\mathcal{F}_e^*.
\end{eqnarray}

\vspace{2mm}
Now, define $\mathcal{F}_e^*\cap L^p(\mu)$, $1<p<\infty$, in the following sense:
\begin{eqnarray}\label{Fe*Lp}
\mathcal{F}_e^*\cap L^p(\mu):=\{v\in L^p(\mu)|\exists C\in(0,\infty)~\text{such~that}~\mu(uv)\leq C\|u\|_{\mathcal{F}_e},~\forall u\in\mathcal{F}_e\cap L^{\frac{p}{p-1}}(\mu)\},
\end{eqnarray}
equipped with the norm $\|v\|_{L^p(\mu)}+\|v\|_{\mathcal{F}_e^*}$. Since $(\mathcal{E},D(\mathcal{E}))$ is a transient Dirichlet space, from \cite[page 131, Proposition 3.1]{RRW}, we know that $\mathcal{F}_e\cap L^{\frac{p}{p-1}}(\mu)$ is a dense subset of both $\mathcal{F}_e$ and $L^{\frac{p}{p-1}}(\mu)$, which in particular implies that $\mathcal{F}_e^*\cap L^p(\mu)$ is really a subset of $\mathcal{F}_e^*$ and we have
\begin{eqnarray*}\label{v2}
_{\mathcal{F}_e^*}\langle v,u\rangle_{\mathcal{F}_e}=\mu(uv),~\forall u\in \mathcal{F}_e\cap L^{\frac{p}{p-1}}(\mu),~v\in \mathcal{F}_e^*\cap L^p(\mu).
\end{eqnarray*}

The following lemma was proved in \cite{RRW}, the part (ii) here is a special case of \cite[Lemma 3.3 (ii)]{RRW}, (iii) is a special case of \cite[Lemma 3.3 (iii)]{RRW}. We shall heavily use this lemma, especially \eqref{rrw3}, in the proof of Theorem \ref{theorem}.
\begin{lemma}\label{RRW}
(i)~The map $\bar{L}:\mathcal{F}_e\rightarrow\mathcal{F}_e^*$ defined by
\begin{eqnarray}\label{rrw1}
\bar{L}v:=-\mathcal{E}(v,\cdot),~~v\in\mathcal{F}_e
\end{eqnarray}
(i.e. the Riesz isomorphism of $\mathcal{F}_e$ and $\mathcal{F}_e^*$ multiplied by (-1)) is the unique continuous linear extension of the map
\begin{eqnarray*}
D(L)\ni v\mapsto\mu(Lv\cdot)\in\mathcal{F}_e^*.
\end{eqnarray*}
(ii)~Let $u\in \mathcal{F}_e\cap L^{\frac{p}{p-1}}(\mu)$, $v\in \mathcal{F}_e^*\cap L^p(\mu)$. Then
\begin{eqnarray}\label{rrw3}
\langle \bar{L}u,v\rangle_{\mathcal{F}_e^*}=-\mu(uv).
\end{eqnarray}
\end{lemma}

If there is no danger of confusion, below we write $L$ instead of $\bar{L}$. For a Hilbert space $\mathbb{H}$, throughout the paper, let $L^2([0,T]\times\Omega;\mathbb{H})$ denote the space of all $\mathbb{H}$-valued square-integrable functions on
$[0,T]\times\Omega$, $C([0,T];\mathbb{H})$ the space of all
continuous $\mathbb{H}$-valued functions on $[0,T]$, and $L^\infty([0,T];\Bbb{H})$ the space of all $\Bbb{H}$-valued essentially bounded measurable functions on $[0,T]$. For two
Hilbert spaces $H_1$ and $H_2$, the space of Hilbert-Schmidt
operators from $H_1$ to $H_2$ is denoted by $L_2(H_1, H_2)$. For
simplicity, the positive constants $C$, $C_1$, $C_2$ and $C_3$ used in this paper may change from line to line.

\section{Assumptions and main results}\label{Assumptions and main results}

Recall that in \eqref{eq:1}, we set $\beta:=\partial \psi$. Let $K:=L^1(\mu)\cap L^\infty(\mu)\cap \mathcal{F}^*_e$. Besides assumption \textbf{(T)}, we need the following assumptions:

\vspace{2mm}

\textbf{(H1)}~~$\psi:\Bbb{R}\rightarrow [0,\infty)$ is a convex, lower-semicontinuous function with $\psi(0)=0$.

\textbf{(H2)}
\begin{eqnarray}\label{psi}
\lim_{|r|\rightarrow\infty}\frac{\psi(r)}{|r|}=+\infty.
\end{eqnarray}

\textbf{(H3)}~~There exists $0<c<\infty$ such that
\begin{eqnarray}\label{beta}
\inf\{|\eta|:\eta\in \beta(r)\}\leq c(|r|+1),~~ \forall r\in\Bbb{R}.
\end{eqnarray}

\textbf{(H4)} $B: [0, T]\times K\times
\Omega\rightarrow L_2(U, F^*_{1,2})$ is progressively
measurable, i.e. for any $t\in[0,T]$, this mapping restricted to
$[0,t]\times K\times \Omega$ is measurable w.r.t.
$\mathcal{B}([0,t])\times\mathcal{B}(K)\times \mathcal{F}_t$,
where $\mathcal{B}(\cdot)$ is the Borel $\sigma$-field for a
topological space. $B(t,u)$ satisfies:

\vspace{2mm}

\noindent \textbf{(i)}~~There exists $C_1\in[0, \infty)$ such that for all $\nu\in(0,1]$,
$$\|B(\cdot, u)-B(\cdot, v)\|^2_{L_2(U, F^*_{1,2,\nu})}\leq C_1\|u-v\|^2_{F^*_{1,2,\nu}},~~ \forall u, v\in K\ \text{on}\ [0, T]\times \Omega,$$
where for simplicity, we write $B(t,u)$
meaning the mapping $\omega\mapsto B(t,u,\omega)$.

\noindent \textbf{(ii)}~~There exists $C_2\in(0, \infty)$ such that for all $\nu\in(0,1]$,
$$\|B(\cdot, u)\|^2_{L_2(U, F^*_{1,2,\nu})}\leq C_2(\|u\|^2_{F^*_{1,2,\nu}}+1),~~ \forall u\in K\ \text{on}\ [0, T]\times \Omega.$$

\noindent \textbf{(iii)}~~There exists $C_3\in(0, \infty)$ such that
$$\|B(\cdot, u)\|^2_{L_2(U, L^2(\mu))}\leq C_3(|u|^2_2+1),~~ \forall u\in K\ \text{on}\ [0, T]\times \Omega.$$

\begin{remark}\label{remark}
\textbf{(i)} \eqref{lim1} and \textbf{(H4)(i)}, \textbf{(ii)} imply that,
\begin{eqnarray}\label{fe1}
\|B(\cdot, u)-B(\cdot, v)\|^2_{L_2(U, \mathcal{F}^*_e)}\leq C_1\|u-v\|^2_{\mathcal{F}^*_e},~~ \forall u,v\in K~ \text{on}~ [0, T]\times \Omega.
\end{eqnarray}
\begin{eqnarray}\label{fe2}
\|B(\cdot, u)\|^2_{L_2(U, \mathcal{F}^*_e)}\leq C_2(\|u\|^2_{\mathcal{F}^*_e}+1),~~ \forall u\in K\ \text{on}\ [0, T]\times \Omega.
\end{eqnarray}
\end{remark}

\begin{remark}\label{H1 and H3}
From \cite[page:47, Theorem 2.8]{Barbu}, we know that if $\widetilde{\psi}:\Bbb{R}\rightarrow(-\infty,+\infty]$ is a lower semicontinuous proper convex function, then its subdifferntial (\cite[page 7, (1.15)]{Barbu}) denoted by $\widetilde{\beta}:\mathbb{R}\rightarrow 2^\mathbb{R}$, i.e.,
$$\widetilde{\beta}(x):=\partial\widetilde{\psi}(x)=\{y\in\Bbb{R};\widetilde{\psi}(x)\leq \widetilde{\psi}(u)+\langle y,x-u\rangle,~\forall u\in\Bbb{R}\},$$
is a maximal monotone graph. Therefore, \textbf{(H1)} implies this strictly weak condition\\
\textbf{(H1)'}~~$\beta: \mathbb{R}\rightarrow 2^\mathbb{R}$ is a maximal monotone graph such that $0\in\beta(0)$.
\end{remark}

\begin{example}\label{example FDE}
Typical examples for $\psi$ which satisfies \textbf{(H1)} and \textbf{(H2)}, and for $\beta$ satisfying \textbf{(H3)} are
\begin{eqnarray*}
\psi(r):=\frac{1}{\theta+1}|r|^{\theta+1},\ \forall r\in\Bbb{R},
\end{eqnarray*}
where $\theta\in(0,1)$, and
\begin{eqnarray*}
\beta(r)=\partial\psi(r)=|r|^{\theta-1}r,\ \forall r\in \Bbb{R}.
\end{eqnarray*}
In this case, \eqref{SPME1} is the stochastic FDE perturbed by multiplicative noise. Applications for FDEs include for instance: plasma diffusion with the Okuda-Dawson scaling leads to the FDE with $\theta:=\frac{1}{2}$ (\cite{BH78}); if $\theta:=\frac{d-2}{d+2}$, $d\geq3$, it is related to the famous Yamabe flow of Riemannian geometry (see \cite[Section 7.5]{V}); in the case that $0<\theta<1$, King studied it in a model of diffusion of impurities in silicon (\cite{Ki88}). One property of the solutions to FDEs is that they decay to zero in finite time with positive probability. We refer to \cite{V} for more details on FDEs, \cite{GRSIAM, RRW} and the references therein for more details on stochastic FDEs.
\end{example}
\begin{example}\label{example Zhang}
Other examples for $\psi$ which satisfies \textbf{(H1)} and \textbf{(H2)}, and for $\beta$ satisfying \textbf{(H3)} are
\begin{eqnarray*}
\psi(r):=\left\{
           \begin{array}{ll}
             &\!\!\!\!\!\!\frac{1}{2}r^2+r,~~~~~~~~~~~~\text{if}~r>0,\\
             &\!\!\!\!\!\!0,\ \text{if}~r\leq0,
           \end{array}
         \right.
\end{eqnarray*}
and
\begin{eqnarray*}\label{zhang beta}
\beta(r)=\partial\psi(r)=\left\{
           \begin{array}{ll}            &\!\!\!\!\!\!r+1,~~~~~~~~~~~~\text{if}~r>0,\label{zhang}\\
             &\!\!\!\!\!\![0,1],~~~\text{if}~r=0,\\
             &\!\!\!\!\!\!0,~~~~~~~~~~~~\text{if}~r<0.
           \end{array}
         \right.
\end{eqnarray*}
In this case, \eqref{SPME1} is related to the Zhang model (\cite{Zhang}), which is a modified version of the Bak-Tang-Wiesenfeld (BTW) model introduced in \cite{BTW88}. Both models are cellular automaton models that describe SOC. SOC models have applications in various dynamical systems, e.g., earthquake mechanisms, forest fires, and sandpiles, etc. These systems have a critical state in which small perturbations can lead to large-scale events, and SOC models attempt to explain this behavior. We refer to \cite{BGS,GCMP,BDRCMP2012,BDRCMP} and the references therein for more studies of SOC model.
\end{example}

To describe the SVI solutions, we introduce the energy functional $\varphi:\mathcal{F}_e^*\rightarrow [0,\infty]$ by
\begin{eqnarray}\label{convex}
\varphi(v):=\left\{
           \begin{array}{ll}
             &\!\!\!\!\!\!\!\! \int_E \psi(v(x))\mu(dx),~~~~~ v\in L^1(\mu)\cap\mathcal{F}_e^*,\\
             &\!\!\!\!\!\!\!\!+\infty,~~~~~~~~~~~~~~~~~ ~~~~\text{else}.
           \end{array}
         \right.
\end{eqnarray}

\begin{definition}\label{svi}
Let $x_0\in L^2(\Omega,\mathcal{F}_0;\mathcal{F}_e^*)$. An $\mathcal{F}_t$-adapted process $X\in L^2(\Omega;C([0,T];\mathcal{F}^*_e))$ is said to be an SVI solution to \eqref{eq:1} if the followings hold:

(i) (Regularity)
\begin{eqnarray}\label{svii}
\varphi(X)\in L^1([0,T]\times\Omega).
   \end{eqnarray}

(ii) (SVI (stochastic variational inequality)) For each $\mathcal{F}_t$-progressively measurable process $G\in L^2([0,T]\times \Omega;\mathcal{F}^*_e)$, and each $\mathcal{F}_t$-adapted process $Z\in L^2(\Omega;C([0,T];\mathcal{F}_e^*))\cap L^2([0,T]\times \Omega; L^2(\mu))$ solving the equation
\begin{eqnarray}\label{svi2}
Z_t:=Z_0+\int_0^tG_sds+\int_0^tB(s,Z_s)dW_s,~~\forall t\in[0,T],
\end{eqnarray}
we have that for some $C>0$,
\begin{eqnarray}\label{svi3}
&&\!\!\!\!\!\!\!\!\Bbb{E}\|X_t-Z_t\|^2_{\mathcal{F}^*_e}+2\Bbb{E}\int_0^t\varphi(X_s)ds\nonumber\\
\leq&&\!\!\!\!\!\!\!\!\Bbb{E}\|x_0-Z_0\|_{\mathcal{F}^*_e}+2\Bbb{E}\int_0^t\varphi(Z_s)ds-2\Bbb{E}\int_0^t\langle G_s,X_s-Z_s\rangle_{\mathcal{F}^*_e}ds\nonumber\\
&&\!\!\!\!\!\!\!\!+C\Bbb{E}\int_0^t\|X_s-Z_s\|^2_{\mathcal{F}^*_e}ds, ~~\forall t\in[0,T].
\end{eqnarray}
\end{definition}

\vspace{2mm}
The following theorem is the main result of this paper, i.e., the well-posedness of \eqref{eq:1} in the sense of Definition \ref{svi}.
\begin{theorem}\label{theorem}
Let $x_0\in L^2(\Omega,\mathcal{F}_0;\mathcal{F}_e^*)$. Suppose \textbf{(H1)}-\textbf{(H4)} and \textbf{(T)} are satisfied, then there exists an SVI solution $X$ to \eqref{eq:1} in the sense of Definition \ref{svi}. Furthermore, if
\begin{eqnarray}\label{uniqueness assumption}
\|P_t\|_{1\rightarrow2}<\infty,\ \forall t>0,
\end{eqnarray}
then for two SVI solutions $X$, $Y$ with initial conditions $x_0, y_0\in L^2(\Omega,\mathcal{F}_0;\mathcal{F}_e^*)$, we have
\begin{eqnarray}\label{th1}
\sup_{t\in [0,T]}\Bbb{E}\|X_t-Y_t\|^2_{\mathcal{F}_e^*} \leq C\Bbb{E}\|x_0-y_0\|^2_{\mathcal{F}_e^*}.
\end{eqnarray}
\end{theorem}

\begin{remark}\label{ultraboundeded}
\textbf{Explanation of \eqref{uniqueness assumption}}: Throughout the paper we use $\|P_t\|_{p\rightarrow q}$ to denote the norm of $P_t$ from $L^p(\mu)$ to $L^q(\mu)$. Recall from \cite[page 154, Definition 3.3.1]{Wang} that $P_t$ is ultrabounded means $\|P_t\|_{2\rightarrow\infty}<\infty$, $\forall t>0$. But since $P_t$ is symmetric, we have, $\|P_t\|^2_{2\rightarrow\infty}=\|P_t\|^2_{1\rightarrow2}=\|P_{2t}\|_{1\rightarrow\infty}$, $t>0$, (cf. the proof of \cite[Theorem 3.3.15]{Wang}). In particular, our results apply to $-f(-L)$, where $f$ is a Bernstein function as in e.g. Examples \ref{Bernstein function proposition 1} and \ref{Bernstein function proposition 2}.
\end{remark}

To prove Theorem \ref{theorem}, we need the existence of strong solutions to the following approximating equations for \eqref{eq:1}.
\begin{eqnarray}\label{eq:3}
\left\{
  \begin{array}{ll}
    dX^\varepsilon_t=L\big(\beta^\varepsilon(X^\varepsilon_t)+\varepsilon X^\varepsilon_t\big)dt+B(t,X^\varepsilon_t)dW_t, & \hbox{} \\
    X^\varepsilon_0=x_0, & \hbox{}
  \end{array}
\right.
\end{eqnarray}
with $0<\varepsilon<1$, $x_0\in L^2(\Omega,\mathcal{F}_0;L^2(\mu)\cap\mathcal{F}_e^*)$. Here $\beta^\varepsilon:\Bbb{R}\rightarrow\Bbb{R}$ is the Yosida approximation of $\beta$ (See Appendix \ref{Moreau Yosida} for more properties of $\beta^\varepsilon$). We have the following Proposition for \eqref{eq:3}.
\begin{proposition}\label{theorem2}
Suppose \textbf{(H1)'}, \textbf{(H3)}, \textbf{(H4)} and \textbf{(T)} are satisfied. Let $\varepsilon\in(0,1)$, $x_0\in L^2(\Omega,\mathcal{F}_0;L^2(\mu)\cap\mathcal{F}_e^*)$. Then there exists an $\mathcal{F}_t$-adapted strong solution to \eqref{eq:3} such that
\begin{eqnarray}\label{3.1}
X^\varepsilon\in L^2(\Omega; C([0,T];\mathcal{F}_e^*))\cap L^2\big([0,T]\times\Omega;L^2(\mu)),
\end{eqnarray}
\begin{equation}\label{3.2}
\int_0^\cdot\big(\beta^\varepsilon(X^\varepsilon_s)+\varepsilon X^\varepsilon_s\big)ds\in C([0,T];\mathcal{F}_e),~~\Bbb{P}-a.s.,
\end{equation}
\begin{equation}\label{3.3}
X^\varepsilon_t=x_0+L\int_0^t\big(\beta^\varepsilon(X^\varepsilon_s)+\varepsilon X^\varepsilon_s\big)ds+\int_0^tB(s,X^\varepsilon_s)dW_s,~\text{holds\ in}\ \mathcal{F}_e^*,~\Bbb{P}-a.s.,~ \forall t\in[0,T].
\end{equation}
Moreover, there exists a constant $C\in(0,\infty)$ which is independent of $\varepsilon$ such that
\begin{eqnarray}\label{3.4}
&& \Bbb{E}\Big[\sup_{t\in[0,T]}|X^\varepsilon_t|^2_2\Big]+\varepsilon\Bbb{E}\int_0^T\|X^\varepsilon_s\|^2_{F_{1,2}}ds\leq C\big(\Bbb{E}|x_0|^2_2+1\big).
\end{eqnarray}
\end{proposition}

\begin{proof}
\eqref{3.1}-\eqref{3.3} follow from \cite[Theorem 3.2]{RWX2021}, in which we take $m:=1$ and $\mu(E)<\infty$. To prove \eqref{3.4}, consider the following equations for \eqref{eq:3} used in \cite[Theorem 3.2]{RWX2021}
\begin{eqnarray*}\label{eq:4}
\left\{
  \begin{array}{ll}
    dX^{\varepsilon,\nu}_t+(\nu-L)\big(\beta^\varepsilon(X^{\varepsilon,\nu}_t+\varepsilon X^{\varepsilon,\nu}_t)dt=B(t,X^{\varepsilon,\nu}_t)dW_t, & \hbox{} \\
    X^{\varepsilon,\nu}_0=x_0\in L^2(\Omega,\mathcal{F}_0;L^2(\mu)), & \hbox{}
  \end{array}
\right.
\end{eqnarray*}
with $\varepsilon\in(0,1)$, $\nu\in(0,1]$. Recall from \cite{RWX2022} that $L^2(\mu)\subset F_{1,2}^*$ continuously and densely (cf. four lines before \cite[(2.5)]{RWX2022}). From \cite[Claim 3.1]{RWX2022} we know that there exists a constant $C>0$ which is independent of $\varepsilon\in(0,1)$, $\nu\in(0,1]$ such that
\begin{eqnarray}\label{lemma2.9}
\Bbb{E}\Big[\sup_{t\in[0,T]}|X^{\varepsilon,\nu}_t|_2^2\Big]+\varepsilon\Bbb{E}\int_0^T\|X^{\varepsilon,\nu}_s\|^2_{F_{1,2}}ds\leq C\big(\Bbb{E}|x_0|_2^2+1\big),
\end{eqnarray}
so we may extract a weakly (weak* respectively) convergent subsequence (for simplicity we stick with the same notation $X^{\varepsilon,\nu}$), and there exists an $\widetilde{X^{\varepsilon}}\in L^2([0,T]\times\Omega;F_{1,2})\cap L^2(\Omega;L^\infty([0,T];L^2(\mu)))$ such that as $\nu\rightarrow0$,
\begin{eqnarray*}
&&\!\!\!\!\!\!\!\!X^{\varepsilon,\nu}\rightharpoonup \widetilde{X^{\varepsilon}}, ~\text{in}~L^2([0,T]\times\Omega;F_{1,2})\subset L^2([0,T]\times\Omega;L^2(\mu))\subset L^2([0,T]\times\Omega;F^*_{1,2})\nonumber\\
&&\!\!\!\!\!\!\!\!X^{\varepsilon,\nu}\rightharpoonup^* \widetilde{X^{\varepsilon}}, ~\text{in}~L^2(\Omega;L^\infty([0,T];L^2(\mu)))\subset L^2(\Omega;L^2([0,T];L^2(\mu)))\subset L^2([0,T]\times\Omega;F^*_{1,2}).
\end{eqnarray*}
But from the proof of \cite[Theorem 3.2]{RWX2021} (cf. arguments around \cite[(4.52)]{RWX2021}), we know that $X^{\varepsilon,\nu}\rightarrow X^{\varepsilon}$ in $L^2(\Omega;C([0,T];F^*_{1,2}))$, hence also in $L^2(\Omega;L^2([0,T];F^*_{1,2}))\subset L^2([0,T]\times\Omega;F^*_{1,2})$, therefore, $\widetilde{X^{\varepsilon}}=X^{\varepsilon}$, $dt\otimes d\Bbb{P}\otimes d\mu$-a.s.. By weak lower semicontinuity of the norms we may pass to the limit in \eqref{lemma2.9} and get \eqref{3.4} as claimed.
\end{proof}

\begin{remark}\label{drop of (H4) in EJP}
We can drop the assumptions (H3)(ii) and (H4) in \cite{RWX2021}, because (H3)(ii) is used in the proof of \cite[Lemma 4.3]{RWX2021}, while (H4) is used in the proof of \cite[Lemma 4.4]{RWX2021} (cf. (4.34), (4.40) and (4.45) in \cite{RWX2021}), which are not needed in our paper.
\end{remark}

\begin{remark}\label{linear growth condition of B}
The right hand-sides of \textbf{(H4)(ii)} and \textbf{(H4)(iii)} look different from the assumptions (H2)(ii) and (H3)(i) in \cite[Theorem 3.2]{RWX2021}, where there is no constant term. However, by using similar ideas as in the proofs of \cite[Theorem 3.2]{RWX2021} and \cite[Claim 3.1]{RWX2022}, it is not difficult to get Proposition \ref{theorem2}, see e.g. \cite[Lemma 4.1]{WZJDE} and \cite[Proposition 4.1]{WZJEE}, which are follow-up works of \cite{RWX2022}, for the treatments of the linear growth conditions as in this paper. So, here we still regard \textbf{(H4)(ii)} and \textbf{(H4)(iii)} as the corresponding assumptions in \cite[Theorem 3.2]{RWX2021}.
\end{remark}

\begin{remark}\label{interchange order}
By the property of the Bochner integral (see e.g. \cite[page 212, Proposition A.2.2]{LR}), \eqref{3.4} allows us to interchange the order of $L$ and $\int_0^t~\cdot~ ds$ in \eqref{3.3}, i.e., we have
\begin{equation}\label{3.3.3}
X^\varepsilon_t=x_0+\int_0^tL\big(\beta^\varepsilon(X^\varepsilon_s)+\varepsilon X^\varepsilon_s\big)ds+\int_0^tB(s,X^\varepsilon_s)dW_s,~\text{holds\ in}\ \mathcal{F}_e^*,~\Bbb{P}-a.s.,~ \forall t\in[0,T].
\end{equation}
\end{remark}

\begin{proof}\textbf{(Proof of Theorem \ref{theorem})}
\\
\textbf{(Convergence)}~~From \eqref{3.3.3} we know that for two strong solutions $X^{\varepsilon_1}$, $X^{\varepsilon_2}$ to \eqref{eq:3} with initial conditions $x_0^1, x_0^2\in L^2(\Omega,\mathcal{F}_0;L^{2}(\mu)\cap\mathcal{F}_e^*)$,
\begin{eqnarray*}
&&\!\!\!\!\!\!\!\!X^{\varepsilon_1}_t-X^{\varepsilon_2}_t\nonumber\\
=&&\!\!\!\!\!\!\!\!x_0^1-x_0^2+\int_0^tL\big(\beta^{\varepsilon_1}(X^{\varepsilon_1}_s)-\beta^{\varepsilon_2}(X^{\varepsilon_2}_s)\big)ds+\int_0^tL\big(\varepsilon_1 X^{\varepsilon_1}_s-\varepsilon_2 X^{\varepsilon_2}_s\big)ds\nonumber\\
&&\!\!\!\!\!\!\!\!+\int_0^t\big(B(s,X^{\varepsilon_1}_s)-B(s,X^{\varepsilon_2}_s)\big)dW_s,\ \text{holds\ in}\ \mathcal{F}_e^*,\ \Bbb{P}-a.s.,\ \forall t\in[0,T].
\end{eqnarray*}
Applying It\^{o}'s formula to $e^{-Kt}\|X^{\varepsilon_1}_t-X^{\varepsilon_2}_t\|^2_{\mathcal{F}_e^*}$ with $K\in[0,\infty)$, we observe that $\forall t\in [0,T]$,
\begin{eqnarray}\label{main}
&&\!\!\!\!\!\!\!\!e^{-Kt}\|X^{\varepsilon_1}_t-X^{\varepsilon_2}_t\|^2_{\mathcal{F}_e^*}\nonumber\\
=&&\!\!\!\!\!\!\!\!\|x_0^1-x_0^2\|^2_{\mathcal{F}_e^*}-K\int_0^te^{-Ks}\|X^{\varepsilon_1}_s-X^{\varepsilon_2}_s\|^2_{\mathcal{F}_e^*}ds\nonumber\\
&&\!\!\!\!\!\!\!\!+2\int_0^te^{-Ks}\langle L\big(\beta^{\varepsilon_1}(X^{\varepsilon_1}_s)-\beta^{\varepsilon_2}(X^{\varepsilon_2}_s)\big), X^{\varepsilon_1}_s-X^{\varepsilon_2}_s\rangle_{\mathcal{F}_e^*}ds\nonumber\\
&&\!\!\!\!\!\!\!\!+2\int_0^te^{-Ks}\langle L\big(\varepsilon_1 X^{\varepsilon_1}_s-\varepsilon_2 X^{\varepsilon_2}_s\big), X^{\varepsilon_1}_s-X^{\varepsilon_2}_s\rangle_{\mathcal{F}_e^*}ds\nonumber\\
&&\!\!\!\!\!\!\!\!+2\int_0^te^{-Ks}\langle X^{\varepsilon_1}_s-X^{\varepsilon_2}_s, \big(B(s,X^{\varepsilon_1}_s)-B(s,X^{\varepsilon_2}_s)\big)dW_s\rangle_{\mathcal{F}_e^*}\nonumber\\
&&\!\!\!\!\!\!\!\!+\int_0^te^{-Ks}\|B(s,X^{\varepsilon_1}_s)-B(s,X^{\varepsilon_2}_s)\|^2_{L_2(U,\mathcal{F}_e^*)}ds.
\end{eqnarray}
For the fourth term in the right hand-side of \eqref{main}, from \eqref{3.1} we know that $X^{\varepsilon_1}, X^{\varepsilon_2}\in L^2(\Omega;C([0,T];\mathcal{F}_e^*))\cap L^2(\Omega\times[0,T];L^2(\mu))$, and from \eqref{3.4} we know that $\varepsilon_1X^{\varepsilon_1}, \varepsilon_2X^{\varepsilon_2}\in L^2(\Omega\times[0,T];F_{1,2})$. Hence by \eqref{rrw3} with $p=2$,
\begin{eqnarray*}\label{l2}
&&\!\!\!\!\!\!\!\!\langle L\big(\varepsilon_1 X^{\varepsilon_1}_s-\varepsilon_2 X^{\varepsilon_2}_s\big), X^{\varepsilon_1}_s-X^{\varepsilon_2}_s\rangle_{\mathcal{F}_e^*}\nonumber\\
=&&\!\!\!\!\!\!\!\!-\int_E\big(\varepsilon_1 X^{\varepsilon_1}_s-\varepsilon_2 X^{\varepsilon_2}_s\big)\cdot\big(X^{\varepsilon_1}_s-X^{\varepsilon_2}_s\big)d\mu\nonumber\\
\leq&&\!\!\!\!\!\!\!\! 2(\varepsilon_1+\varepsilon_2)\big(|X^{\varepsilon_1}_s|^{2}_{2}+|X^{\varepsilon_2}_s|^{2}_{2}\big),~~~ds\otimes\Bbb{P}-a.e..
\end{eqnarray*}
For the third term in the right hand-side of \eqref{main}, by \eqref{3.4} and \cite[page 35, Proposition 4.11]{MR} we know that $\beta^{\varepsilon_1}(X^{\varepsilon_1}_s), \beta^{\varepsilon_2}(X^{\varepsilon_2}_s)\in L^2(\Omega\times[0,T];F_{1,2})$, then by \eqref{3.1}, \eqref{rrw3} with $p=2$ and \eqref{beta3},
\begin{eqnarray*}\label{l1}
&&\!\!\!\!\!\!\!\!\langle L\big(\beta^{\varepsilon_1}(X^{\varepsilon_1}_s)-\beta^{\varepsilon_2}(X^{\varepsilon_2}_s)\big), X^{\varepsilon_1}_s-X^{\varepsilon_2}_s\rangle_{\mathcal{F}_e^*}\nonumber\\
=&&\!\!\!\!\!\!\!\!-\int_E\big(\beta^{\varepsilon_1}(X^{\varepsilon_1}_s)-\beta^{\varepsilon_2}(X^{\varepsilon_2}_s)\big)\cdot\big(X^{\varepsilon_1}_s-X^{\varepsilon_2}_s\big)d\mu\nonumber\\
\leq&&\!\!\!\!\!\!\!\! C(\varepsilon_1+\varepsilon_2)\big(|X^{\varepsilon_1}_s|^{2}_{2}+|X^{\varepsilon_2}_s|^{2}_{2}+\mu(E)\big),~~~ds\otimes\Bbb{P}-a.e..
\end{eqnarray*}
Thus, $\forall t\in [0,T]$, we have
\begin{eqnarray*}
&&\!\!\!\!\!\!\!\!e^{-Kt}\|X^{\varepsilon_1}_t-X^{\varepsilon_2}_t\|^2_{\mathcal{F}_e^*}\nonumber\\
=&&\!\!\!\!\!\!\!\!\|x_0^1-x_0^2\|^2_{\mathcal{F}_e^*}-K\int_0^te^{-Ks}\|X^{\varepsilon_1}_s-X^{\varepsilon_2}_s\|^2_{\mathcal{F}_e^*}ds\nonumber\\
&&\!\!\!\!\!\!\!\!+C(\varepsilon_1+\varepsilon_2)\int_0^te^{-Ks}\Big(|X^{\varepsilon_1}_s|^{2}_{2}+|X^{\varepsilon_2}_s|^{2}_{2}+\mu(E)\Big)ds\nonumber\\
&&\!\!\!\!\!\!\!\!+2\int_0^te^{-Ks}\langle X^{\varepsilon_1}_s-X^{\varepsilon_2}_s, \big(B(s,X^{\varepsilon_1}_s)-B(s,X^{\varepsilon_2}_s)\big)dW_s\rangle_{\mathcal{F}_e^*}\nonumber\\
&&\!\!\!\!\!\!\!\!+\int_0^te^{-Ks}\|B(s,X^{\varepsilon_1}_s)-B(s,X^{\varepsilon_2}_s)\|^2_{L_2(U,\mathcal{F}_e^*)}ds.
\end{eqnarray*}
Using the Burkholder-Davis-Gundy (BDG) inequality for $p=1$, \eqref{fe1} and \eqref{3.4}, we obtain
\begin{eqnarray}\label{app1}
&&\!\!\!\!\!\!\!\!\Bbb{E}\sup_{t\in[0,T]}e^{-Kt}\|X^{\varepsilon_1}_t-X^{\varepsilon_2}_t\|^2_{\mathcal{F}_e^*}\nonumber\\
\leq&&\!\!\!\!\!\!\!\!2\Bbb{E}\|x_0^1-x_0^2\|^2_{\mathcal{F}_e^*}+C_{K,T}(\varepsilon_1+\varepsilon_2)\Big(\Bbb{E}|x_0^1|^{2}_{2}+\Bbb{E}|x_0^2|^{2}_{2}+1\Big),
\end{eqnarray}
for $K>0$ large enough, where $C_{K,T}$ depends on $K$ and $T$, but is independent of $\varepsilon_1$ and $\varepsilon_2$.

Now assume that $x_0^1=x_0^2=x_0\in L^2(\Omega,\mathcal{F}_0; L^2(\mu)\cap\mathcal{F}_e^*)$. Then from \eqref{app1} we have
\begin{eqnarray*}
&&\!\!\!\!\!\!\!\!\Bbb{E}\sup_{t\in[0,T]}e^{-Kt}\|X^{\varepsilon_1}_t-X^{\varepsilon_2}_t\|^2_{\mathcal{F}_e^*}\nonumber\\
\leq&&\!\!\!\!\!\!\!\!C(\varepsilon_1+\varepsilon_2)(\Bbb{E}|x_0|^{2}_{2}+1),
\end{eqnarray*}
where $C$ is independent of $\varepsilon_1$ and $\varepsilon_2$, and thus, by completeness there exists an $\mathcal{F}_t$-adapted process $X(\cdot,x_0)\in L^2(\Omega;C([0,T];\mathcal{F}_e^*))$ such that
\begin{eqnarray}\label{convergence1}
\Bbb{E}\sup_{t\in[0,T]}\|X^{\varepsilon}_t-X(t,x_0)\|^2_{\mathcal{F}_e^*}\rightarrow0,~~\text{as}~~\varepsilon\rightarrow0.
\end{eqnarray}
Then, for $x_0^1, x_0^2\in L^2(\Omega,\mathcal{F}_0;L^2(\mu)\cap\mathcal{F}_e^*)$, taking $\varepsilon_1, \varepsilon_2\rightarrow0$ in \eqref{app1}, we get
\begin{eqnarray}\label{app2}
\Bbb{E}\sup_{t\in[0,T]}e^{-Kt}\|X(t,x_0^1)-X(t,x_0^2)\|^2_{\mathcal{F}_e^*}\leq 2\Bbb{E}\|x_0^1-x_0^2\|^2_{\mathcal{F}_e^*}.
\end{eqnarray}
Then by \eqref{app2}
$$L^2(\mu)\cap\mathcal{F}_e^*\ni x_0\mapsto X(\cdot, x_0)\in L^2(\Omega;C([0,T];\mathcal{F}_e^*))$$
is Lipschitz continuous, hence it has a unique Lipschitz continuous extension from $\mathcal{F}_e^*$ to $L^2(\Omega;C([0,T];\mathcal{F}_e^*))$ which we denote by the the same symbol $X(\cdot,x_0)$, $x_0\in L^2(\Omega,\mathcal{F}_0;\mathcal{F}_e^*)$,
and this process is a candidate for the SVI solution to \eqref{eq:1} with initial value $x_0\in L^2(\Omega,\mathcal{F}_0;\mathcal{F}_e^*)$.


\vspace{3mm}

\textbf{(Regularity)}~~We are going to prove the regularity of $X:=X(\cdot,x_0)$ in the sense of \eqref{svii}, i.e., $\varphi(X)\in L^1([0,T]\times \Omega)$, where $\varphi$ is defined as in \eqref{convex}. Let $X^{\varepsilon,n}_t:=X^\varepsilon(t,x_0^n)$, $t\in[0,T]$, $\varepsilon\in(0,1)$, be strong solutions to \eqref{eq:3} with initial value $x_0^n\in L^2(\Omega,\mathcal{F}_0;L^2(\mu)\cap\mathcal{F}_e^*)$, where $\{x_0^n\}_{n\in\Bbb{N}}$ converges to $x_0$ in $L^2(\Omega,\mathcal{F}_0;\mathcal{F}_e^*)$ as $n\rightarrow\infty$ (such $\{x_0^n\}_{n\in\Bbb{N}}$ exists because $L^2(\mu)\cap\mathcal{F}_e^*$ is dense in $\mathcal{F}_e^*$, hence $L^2(\Omega,\mathcal{F}_0;L^2(\mu)\cap\mathcal{F}_e^*)$ is dense in $L^2(\Omega,\mathcal{F}_0;\mathcal{F}_e^*)$). Denote $X^n_t:=X(t,x_0^n)$, $t\in[0,T]$, the limit of $X^{\varepsilon,n}_t$ in \eqref{convergence1}. Applying It\^{o}'s formula to $e^{-Kt}\|X^{\varepsilon,n}_t\|^2_{\mathcal{F}_e^*}$, by \eqref{3.3.3} we obtain that $\forall t\in [0,T]$,

\begin{eqnarray*}
&&\!\!\!\!\!\!\!\!e^{-Kt}\|X^{\varepsilon,n}_t\|^2_{\mathcal{F}_e^*}\nonumber\\
=&&\!\!\!\!\!\!\!\!\|x_0^n\|^2_{\mathcal{F}_e^*}+2\int_0^te^{-Ks}\langle \varepsilon LX^{\varepsilon,n}_s,X^{\varepsilon,n}_s\rangle_{\mathcal{F}_e^*} ds\nonumber\\
&&\!\!\!\!\!\!\!\!+2\int_0^te^{-Ks}\langle L\beta^\varepsilon(X^{\varepsilon,n}_s),X^{\varepsilon,n}_s\rangle_{\mathcal{F}_e^*} ds+2\int_0^te^{-Ks}\langle X^{\varepsilon,n}_s,B(s,X^{\varepsilon,n}_s)dW_s\rangle_{\mathcal{F}_e^*}\nonumber\\
&&\!\!\!\!\!\!\!\!+2\int_0^te^{-Ks}\|B(s,X^{\varepsilon,n}_s)\|^2_{L_2(U,\mathcal{F}_e^*)}ds-K\int_0^te^{-Ks}\|X^{\varepsilon,n}_s\|^2_{\mathcal{F}_e^*}ds.
\end{eqnarray*}
By \eqref{3.1}, \eqref{3.4} and \eqref{rrw3} with $p=2$, we have
\begin{eqnarray*}
&&\!\!\!\!\!\!\!\!\langle \varepsilon LX^{\varepsilon,n}_s,X^{\varepsilon,n}_s\rangle_{\mathcal{F}_e^*}=-\varepsilon|X^{\varepsilon,n}_s|_2^2\leq0,~~~ds\otimes\Bbb{P}-a.e..
\end{eqnarray*}
For $v\in \mathcal{F}_e^*$, we set
\begin{eqnarray}\label{convexx}
\varphi^\varepsilon(v):=\left\{
           \begin{array}{ll}
             &\!\!\!\!\!\!\!\!\int_E\psi^\varepsilon(v(x))\mu(dx),~~~\text{if}~~ v\in L^{2}(\mu)\cap\mathcal{F}_e^*,\\
             &\!\!\!\!\!\!\!\!+\infty,~~~~~~~~~~~~~~~~~~~~ \text{else}.
           \end{array}
         \right.
\end{eqnarray}
By \eqref{3.4} and \cite[page 35, Proposition 4.11]{MR}, \eqref{3.1}, \eqref{rrw3} with $p=2$, and since $\beta^\varepsilon=\partial \psi^\varepsilon$ and $\psi^\varepsilon$ is convex (see Appendix \ref{Moreau Yosida}), we have
\begin{eqnarray}\label{con}
&&\!\!\!\!\!\!\!\!\langle L\beta^\varepsilon(X^{\varepsilon,n}_s),X^{\varepsilon,n}_s\rangle_{\mathcal{F}_e^*}\nonumber\\
=&&\!\!\!\!\!\!\!\!\int_E\partial\psi^\varepsilon(X^{\varepsilon,n}_s)\cdot(0-X^{\varepsilon,n}_s)d\mu\nonumber\\
\leq&&\!\!\!\!\!\!\!\!-\varphi^\varepsilon(X^{\varepsilon,n}_s), ~~~ds\otimes\Bbb{P}-a.e..
\end{eqnarray}
So, $\forall t\in [0,T]$, we have
\begin{eqnarray*}
&&\!\!\!\!\!\!\!\!e^{-Kt}\|X^{\varepsilon,n}_t\|^2_{\mathcal{F}_e^*}\nonumber\\
\leq&&\!\!\!\!\!\!\!\!\|x_0^n\|^2_{\mathcal{F}_e^*}-2\int_0^te^{-Ks}\varphi^\varepsilon(X^{\varepsilon,n}_s)ds+\int_0^te^{-Ks}\|B(s,X^{\varepsilon,n}_s)\|^2_{L_2(L^2(\mu),\mathcal{F}_e^*)}ds\nonumber\\
&&\!\!\!\!\!\!\!\!+2\int_0^te^{-Ks}\langle X^{\varepsilon,n}_s, B(s,X^{\varepsilon,n}_s)dW_s\rangle_{\mathcal{F}_e^*}-K\int_0^te^{-Ks}\|X^{\varepsilon,n}_s\|^2_{\mathcal{F}_e^*}ds,
\end{eqnarray*}
where the fourth term in the right hand-side of the above inequality is a continuous real-valued martingale (cf. e.g. \cite[Lemma I.0.3]{LR}).
After taking expectation on both sides of the above inequality, by \eqref{fe2}, choosing $K:=C_2$, we get $\forall t\in [0,T]$,
\begin{eqnarray*}
\Bbb{E}e^{-C_2t}\|X^{\varepsilon,n}_t\|^2_{\mathcal{F}_e^*}\leq \Bbb{E}\|x_0^n\|^2_{\mathcal{F}_e^*}-2e^{-C_2t}\Bbb{E}\int_0^t\varphi^\varepsilon(X^{\varepsilon,n}_s)ds+\int_0^te^{-C_2s}C_2ds.
\end{eqnarray*}
Choosing $t=T$ and multiplying both sides of the above inequality by $\frac{1}{2}e^{C_2T}$,
\begin{eqnarray}\label{app3}
\Bbb{E}\int_0^T\varphi^\varepsilon(X^{\varepsilon,n}_s)ds\leq C_T\big(\Bbb{E}\|x_0^n\|^2_{\mathcal{F}_e^*}+1\big)<\infty,
\end{eqnarray}
where $C_T$ is a constant which is dependent on $T$, but independent of $\varepsilon$, $n$.

By \eqref{convex}, \eqref{convexx}, \eqref{yosida3} and \eqref{beta}, we have for $v\in L^2(\mu)\cap\mathcal{F}_e^*$,
\begin{eqnarray}\label{epsilon}
&&\!\!\!\!\!\!\!\!|\varphi^\varepsilon(v)-\varphi(v)|\nonumber\\
\leq&&\!\!\!\!\!\!\!\!\int_E|\psi^\varepsilon(v(x))-\psi(v(x))|\mu(dx)\nonumber\\
\leq&&\!\!\!\!\!\!\!\!\int_E\varepsilon|\beta(v(x))|^2\mu(dx)\nonumber\\
\leq &&\!\!\!\!\!\!\!\!c\varepsilon\int_E(|v(x)|^{2}+1)\mu(dx)=c\varepsilon\big(|v|_{2}^{2}+\mu(E)\big),
\end{eqnarray}
which implies
\begin{eqnarray*}
\varphi(X^{\varepsilon,n}_s)-\varphi^\varepsilon(X^{\varepsilon,n}_s)\leq c\varepsilon\big(|X^{\varepsilon,n}_s|_{2}^{2}+\mu(E)\big),~~~\ ds\otimes \Bbb{P}-a.e.,
\end{eqnarray*}
so,
\begin{eqnarray}\label{estimate4}
\Bbb{E}\int_0^T\varphi^\varepsilon(X^{\varepsilon,n}_s)ds\geq\Bbb{E}\int_0^T\varphi(X^{\varepsilon,n}_s)ds-c\varepsilon\Bbb{E}\int_0^T\big(|X^{\varepsilon,n}_s|_{2}^{2}+\mu(E)\big)ds.
\end{eqnarray}
Since $X^\varepsilon(s,x_0^n)\rightarrow X(s,x_0^n)$ in $L^2(\Omega;C([0,T];\mathcal{F}_e^*))$ as $\varepsilon\rightarrow0$, hence also in $L^2(\Omega\times [0,T];\mathcal{F}_e^*)$. By selecting a subsequence if necessary, $X^\varepsilon(s,x_0^n)\rightarrow X(s,x_0^n)$, $ds\otimes\Bbb{P}$-a.s., in $\mathcal{F}_e^*$. Hence by the lower semicontinuity of $\varphi$ on $\mathcal{F}_e^*$ (cf. Lemma \ref{cvx}), Fatou's Lemma, \eqref{estimate4} and \eqref{3.4}, we conclude
\begin{eqnarray}\label{app4}
&&\!\!\!\!\!\!\!\!\Bbb{E}\int_0^T\varphi(X^{n}_s)ds\nonumber\\
\leq&&\!\!\!\!\!\!\!\!\liminf_{\varepsilon\rightarrow0}\Bbb{E}\int_0^T\varphi(X^{\varepsilon,n}_s)ds\nonumber\\
\leq&&\!\!\!\!\!\!\!\!\liminf_{\varepsilon\rightarrow0}\Big[\Bbb{E}\int_0^T\varphi^\varepsilon(X^{\varepsilon,n}_s)ds+c\varepsilon\Bbb{E}\int_0^T\big(|X^{\varepsilon,n}_s|_{2}^{2}+\mu(E))ds\Big]\nonumber\\
\leq&&\!\!\!\!\!\!\!\!\liminf_{\varepsilon\rightarrow0}\Bbb{E}\int_0^T\varphi^\varepsilon(X^{\varepsilon,n}_s)ds.
\end{eqnarray}
Since $X(s,x_0^n)\rightarrow X(s,x_0)$ in $L^2(\Omega;C([0,T];\mathcal{F}_e^*))$ as $n\rightarrow\infty$, hence also in $L^2(\Omega\times [0,T];\mathcal{F}_e^*)$. By selecting a subsequence if necessary, $X(s,x_0^n)\rightarrow X(s,x_0)$, $ds\otimes\Bbb{P}$-a.s., in $\mathcal{F}_e^*$. Hence  by the lower semicontinuity of $\varphi$ on $\mathcal{F}_e^*$, Fatou's Lemma and taking \eqref{app4}, \eqref{app3} into account, we get
\begin{eqnarray*}\label{regularity estimate}
&&\!\!\!\!\!\!\!\!\Bbb{E}\int_0^T\varphi(X_s)ds\nonumber\\
\leq&&\!\!\!\!\!\!\!\!\liminf_{n\rightarrow\infty}\Bbb{E}\int_0^T\varphi(X^{n}_s)ds\nonumber\\
\leq&&\!\!\!\!\!\!\!\!\liminf_{n\rightarrow\infty}\liminf_{\varepsilon\rightarrow0}\Bbb{E}\int_0^T\varphi^\varepsilon(X^{\varepsilon,n}_s)ds\nonumber\\
\leq&&\!\!\!\!\!\!\!\! C_T(\Bbb{E}\|x_0\|^2_{\mathcal{F}_e^*}+1)<\infty.
\end{eqnarray*}

\textbf{(SVI)}~~ Let $G$, $Z$ be as in Definition \ref{svi} (ii). Applying It\^{o}'s formula to $\|X^{\varepsilon,n}_t-Z_t\|^2_{\mathcal{F}_e^*}$, by \eqref{3.3.3}, $\forall t\in [0,T]$, we have
\begin{eqnarray*}
&&\!\!\!\!\!\!\!\!\Bbb{E}\|X^{\varepsilon,n}_t-Z_t\|^2_{\mathcal{F}_e^*}\nonumber\\
=&&\!\!\!\!\!\!\!\!\Bbb{E}\|x_0^n-Z_0\|^2_{\mathcal{F}_e^*}+2\Bbb{E}\int_0^t\langle\varepsilon LX^{\varepsilon,n}_s+L\beta^\varepsilon(X^{\varepsilon,n}_s)-G_s,X^{\varepsilon,n}_s-Z_s\rangle_{\mathcal{F}_e^*}ds\nonumber\\
&&\!\!\!\!\!\!\!\!+\Bbb{E}\int_0^t\|B(s,X^{\varepsilon,n}_s)-B(s,Z_s)\|^2_{L_2(L^2(\mu),\mathcal{F}_e^*)}ds.
\end{eqnarray*}
Using the convexity of $\psi^\varepsilon$, similarly as to get \eqref{con}, we have
\begin{eqnarray*}
&&\!\!\!\!\!\!\!\!\langle L\beta^\varepsilon(X^{\varepsilon,n}_s),X^{\varepsilon,n}_s-Z_s\rangle_{\mathcal{F}_e^*}\leq\varphi^\varepsilon(Z_s)-\varphi^\varepsilon(X^{\varepsilon,n}_s),~~ds\otimes\Bbb{P}-a.e..
\end{eqnarray*}
By \eqref{rrw3} with $p=2$, H\"{o}lder's inequality and Young's inequality,
\begin{eqnarray*}
&&\!\!\!\!\!\!\!\!\langle\varepsilon LX^{\varepsilon,n}_s,X^{\varepsilon,n}_s-Z_s\rangle_{\mathcal{F}_e^*}\nonumber\\
=&&\!\!\!\!\!\!\!\!-\varepsilon\int_EX^{\varepsilon,n}_s\cdot(X^{\varepsilon,n}_s-Z_s)d\mu\nonumber\\
\leq&&\!\!\!\!\!\!\!\!\varepsilon|X^{\varepsilon,n}_s|_{2}\cdot|X^{\varepsilon,n}_s-Z_s|_{2}\nonumber\\
\leq&&\!\!\!\!\!\!\!\!\frac{1}{2}\varepsilon^{\frac{4}{3}}|X^{\varepsilon,n}_s|_{2}^2+\frac{1}{2}\varepsilon^{\frac{2}{3}}|X^{\varepsilon,n}_s-Z_s|^2_{2},~~ds\otimes\Bbb{P}-a.e..
\end{eqnarray*}
So, by \eqref{fe1}, we have that $\forall t\in [0,T]$,
\begin{eqnarray}\label{lim}
&&\!\!\!\!\!\!\!\!2\Bbb{E}\int_0^t\varphi^\varepsilon(X^{\varepsilon,n}_s)ds\nonumber\\
\leq&&\!\!\!\!\!\!\!\!-\Bbb{E}\|X^{\varepsilon,n}_t-Z_t\|^2_{\mathcal{F}_e^*}+\Bbb{E}\|x_0^n-Z_0\|^2_{\mathcal{F}_e^*}+2\Bbb{E}\int_0^t\varphi^\varepsilon(Z_s)ds\nonumber\\
&&\!\!\!\!\!\!\!\!-2\Bbb{E}\int_0^t\langle G_s,X^{\varepsilon,n}_s-Z_s\rangle_{\mathcal{F}_e^*}ds+C_1\Bbb{E}\int_0^t\|X^{\varepsilon,n}_s-Z_s\|^2_{\mathcal{F}_e^*}ds\nonumber\\
&&\!\!\!\!\!\!\!\!+\Bbb{E}\int_0^t\big(\varepsilon^{\frac{4}{3}}|X^{\varepsilon,n}_s|_{2}^2+\varepsilon^{\frac{2}{3}}|X^{\varepsilon,n}_s-Z_s|^2_{2}\big)ds.
\end{eqnarray}
Since \eqref{yosida2} holds, for $v\in L^2(\mu)\cap \mathcal{F}_e^*$, we have $\varphi^\varepsilon(v)\leq \varphi(v)$. Taking $\liminf_{\varepsilon\rightarrow0}$ in both sides of \eqref{lim}, by \eqref{3.4}, we have that $\forall t\in [0,T]$,
\begin{eqnarray*}
&&\!\!\!\!\!\!\!\!2\liminf_{\varepsilon\rightarrow0}\Bbb{E}\int_0^t\varphi^\varepsilon(X^{\varepsilon,n}_s)ds\nonumber\\
\leq&&\!\!\!\!\!\!\!\!-\Bbb{E}\|X^{n}_t-Z_t\|^2_{\mathcal{F}_e^*}+\Bbb{E}\|x_0^n-Z_0\|^2_{\mathcal{F}_e^*}\nonumber\\
&&\!\!\!\!\!\!\!\!+2\Bbb{E}\int_0^t\varphi(Z_s)ds-2\Bbb{E}\int_0^t\langle G_s,X^{n}_s-Z_s\rangle_{\mathcal{F}_e^*}ds+C_1\Bbb{E}\int_0^t\|X^{n}_s-Z_s\|^2_{\mathcal{F}_e^*}ds.
\end{eqnarray*}
Taking \eqref{app4} into account, we observe that $\forall t\in [0,T]$,
\begin{eqnarray}\label{limm}
&&\!\!\!\!\!\!\!\!2\Bbb{E}\int_0^t\varphi(X^{n}_s)ds\nonumber\\
\leq&&\!\!\!\!\!\!\!\!-\Bbb{E}\|X^{n}_t-Z_t\|^2_{\mathcal{F}_e^*}+\Bbb{E}\|x_0^n-Z_0\|^2_{\mathcal{F}_e^*}\nonumber\\
&&\!\!\!\!\!\!\!\!+2\Bbb{E}\int_0^t\varphi(Z_s)ds-2\Bbb{E}\int_0^t\langle G_s,X^{n}_s-Z_s\rangle_{\mathcal{F}_e^*}ds+C_1\Bbb{E}\int_0^t\|X^{n}_s-Z_s\|^2_{\mathcal{F}_e^*}ds.
\end{eqnarray}
Then taking $n\rightarrow\infty$ in both sides of \eqref{limm}, by the lower semicontinuity of $\varphi$ on $\mathcal{F}_e^*$ (cf. Lemma \ref{cvx}) and Fatou's lemma, we have that $\forall t\in [0,T]$,
\begin{eqnarray*}\label{variational l.s.c.}
&&\!\!\!\!\!\!\!\!2\Bbb{E}\int_0^t\varphi(X_s)ds\nonumber\\
\leq&&\!\!\!\!\!\!\!\!2\Bbb{E}\int_0^t\liminf_{n\rightarrow\infty}\varphi(X^n_s)ds\nonumber\\
\leq&&\!\!\!\!\!\!\!\!2\liminf_{n\rightarrow\infty}\Bbb{E}\int_0^t\varphi(X^n_s)ds\nonumber\\
\leq&&\!\!\!\!\!\!\!\!-\Bbb{E}\|X_t-Z_t\|^2_{\mathcal{F}_e^*}+\Bbb{E}\|x_0-Z_0\|^2_{\mathcal{F}_e^*}\nonumber\\
&&\!\!\!\!\!\!\!\!+2\Bbb{E}\int_0^t\varphi(Z_s)ds-2\Bbb{E}\int_0^t\langle G_s,X_s-Z_s\rangle_{\mathcal{F}_e^*}ds+C_1\Bbb{E}\int_0^t\|X_s-Z_s\|^2_{\mathcal{F}_e^*}ds,
\end{eqnarray*}
which yields \eqref{svi3}.

\textbf{(Uniqueness)}~~Let $X$ be an SVI solution to \eqref{eq:1} and let $Y^{\varepsilon,n}_t:=Y^\varepsilon(t,y_0^n)$,
 $t\in[0,T]$, $\varepsilon\in(0,1)$, be strong solutions to \eqref{eq:3} with initial value $y_0^n\in L^2(\Omega,\mathcal{F}_0;L^2(\mu)\cap\mathcal{F}_e^*)$, where $\{y_0^n\}_{n\in\Bbb{N}}$ converges to $y_0$ in $L^2(\Omega,\mathcal{F}_0;\mathcal{F}_e^*)$, as $n\rightarrow\infty$.

From \eqref{3.1} we know that
$Y^{\varepsilon,n}\in L^2(\Omega;C([0,T];\mathcal{F}_e^*))\cap L^2([0,T]\times\Omega;L^2(\mu))$. By \eqref{rrw1}, the fact that $F_{1,2}\subset \mathcal{F}_e$ continuously and densely and by \eqref{3.4},
\begin{eqnarray*}
\Bbb{E}\int_0^T\|LY^{\varepsilon,n}_s\|^2_{\mathcal{F}_e^*}ds=\Bbb{E}\int_0^T\|Y^{\varepsilon,n}_s\|^2_{\mathcal{F}_e}ds\leq \Bbb{E}\int_0^T\|Y^{\varepsilon,n}_s\|^2_{F_{1,2}}ds<\infty.
\end{eqnarray*}
Furthermore, by \cite[page 35, Proposition 4.11]{MR},
\begin{eqnarray*}
\Bbb{E}\int_0^T\|L\beta^\varepsilon(Y^{\varepsilon,n}_s)\|^2_{\mathcal{F}_e^*}ds=\Bbb{E}\int_0^T\|\beta^\varepsilon(Y^{\varepsilon,n}_s)\|^2_{\mathcal{F}_e}ds<\infty,
\end{eqnarray*}
hence, $LY^{\varepsilon,n}, L\beta^\varepsilon(Y^{\varepsilon,n})\in L^2([0,T]\times\Omega;\mathcal{F}_e^*)$. Therefore, $Z:=Y^{\varepsilon,n}$, $G:=\varepsilon LY^{\varepsilon,n}+L\beta^\varepsilon(Y^{\varepsilon,n})$ are admissible choices for \eqref{svi2}.

Then \eqref{svi3} yields,
\begin{eqnarray*}
&&\!\!\!\!\!\!\!\!\Bbb{E}\|X_t-Y^{\varepsilon,n}_t\|^2_{\mathcal{F}_e^*}+2\Bbb{E}\int_0^t\varphi(X_s)ds\nonumber\\
\leq&&\!\!\!\!\!\!\!\!\Bbb{E}\|x_0-y_0^n\|^2_{\mathcal{F}_e^*}+2\Bbb{E}\int_0^t\varphi(Y^{\varepsilon,n}_s)ds\nonumber\\
&&\!\!\!\!\!\!\!\!-2\Bbb{E}\int_0^t\langle \varepsilon LY^{\varepsilon,n}_s+L\beta^\varepsilon(Y^{\varepsilon,n}_s),X_s-Y^{\varepsilon,n}_s\rangle_{\mathcal{F}_e^*}ds+C\Bbb{E}\int_0^t\|X_s-Y^{\varepsilon,n}_s\|^2_{\mathcal{F}_e^*}ds,\ \forall t\in[0,T].
\end{eqnarray*}
For $v\in L^{2}(\mu)\cap\mathcal{F}_e^*$, by \eqref{3.1}, \eqref{3.4} and \cite[page 35, Proposition 4.11]{MR}, \eqref{rrw3} with $p=2$ and the convexity of $\psi^\varepsilon$, we have
\begin{eqnarray*}
&&\!\!\!\!\!\!\!\!-\langle L\beta^\varepsilon(Y^{\varepsilon,n}_s), v-Y^{\varepsilon,n}_s\rangle_{\mathcal{F}_e^*}\nonumber\\
=&&\!\!\!\!\!\!\!\!\int_E\beta^\varepsilon(Y^{\varepsilon,n}_s)\cdot (v-Y^{\varepsilon,n}_s)d\mu\nonumber\\
=&&\!\!\!\!\!\!\!\!\int_E\partial\psi^\varepsilon(Y^{\varepsilon,n}_s)\cdot (v-Y^{\varepsilon,n}_s)d\mu\nonumber\\
\leq&&\!\!\!\!\!\!\!\!\int_E\psi^\varepsilon(v)-\psi^\varepsilon(Y^{\varepsilon,n}_s)d\mu\nonumber\\
=&&\!\!\!\!\!\!\!\!\varphi^\varepsilon(v)-\varphi^\varepsilon(Y^{\varepsilon,n}_s),~~~ds\otimes\Bbb{P}-a.e.,
\end{eqnarray*}
so,
\begin{eqnarray}\label{1}
-\langle L\beta^\varepsilon(Y^{\varepsilon,n}_s), v-Y^{\varepsilon,n}_s\rangle_{\mathcal{F}_e^*}+\varphi^\varepsilon(Y^{\varepsilon,n}_s)\leq\varphi^\varepsilon(v),~~~ds\otimes\Bbb{P}-a.e..
\end{eqnarray}
By \eqref{epsilon}, we have
\begin{eqnarray}\label{2}
\varphi(Y^{\varepsilon,n}_s)\leq \varepsilon c\big(|Y^{\varepsilon,n}_s|_{2}^{2}+\mu(E)\big)+\varphi^\varepsilon(Y^{\varepsilon,n}_s),~~~ds\otimes\Bbb{P}-a.e..
\end{eqnarray}
Hence, summing the left and right hand-sides of \eqref{2} and \eqref{1} respectively, by \eqref{yosida2}, we get
\begin{eqnarray}\label{eq:fe former}
&&\!\!\!\!\!\!\!\!-\langle L\beta^\varepsilon(Y^{\varepsilon,n}_s), v-Y^{\varepsilon,n}_s\rangle_{\mathcal{F}_e^*}+\varphi(Y^{\varepsilon,n}_s)\nonumber\\
\leq&&\!\!\!\!\!\!\!\!\varphi^\varepsilon(v)+\varepsilon c\big(|Y^{\varepsilon,n}_s|_{2}^{2}+\mu(E)\big)\nonumber\\
\leq&&\!\!\!\!\!\!\!\!\varphi(v)+\varepsilon c\big(|Y^{\varepsilon,n}_s|_{2}^{2}+\mu(E)\big),~~ds\otimes\Bbb{P}-a.e..
\end{eqnarray}

Let $\{\varepsilon_k\}_{k\in\Bbb{N}}$ be a subsequence of $(0,1)$ such that $\lim_{k\rightarrow\infty}\varepsilon_k=0$. Define $$A_1:=\{(s,\omega)|Y^{\varepsilon_k,n}_s(\omega)\in F_{1,2}\cap\mathcal{F}_e^*,~\forall \varepsilon_k,~k\in\Bbb{N},~\forall n\in\Bbb{N}\}.$$
Then $A_1^c$, which is the complement of $A_1$, is a measure zero subset of $[0,T]\times\Omega$. For simplicity, below we stick with the notation $\varepsilon$ for $\varepsilon_k$. Define
$$
A_2:=\{(s,\omega)|\varphi(X_s(\omega))<\infty\}.
$$
Then $A_2^c$ is $dt\otimes\Bbb{P}$ measure zero.

Assume that \eqref{uniqueness assumption} is satisfied. Let $(s,\omega)\in A_1\cap A_2$, then by Proposition \ref{construction of un} below, there exists a sequence $v_m(s,\omega)\in L^2(\mu)\cap\mathcal{F}_e^*$ such that $v_m(s,\omega)\rightharpoonup X_s(\omega)$ in $\mathcal{F}_e^*$ and $\varphi(X_s(\omega))=\lim_{m\rightarrow\infty}\varphi(v_m(s,\omega))$.
Hence, by \eqref{eq:fe former} we have
\begin{eqnarray*}\label{eq:fe}
&&\!\!\!\!\!\!\!\!-\langle L\beta^\varepsilon(Y^{\varepsilon,n}_s), X_s-Y^{\varepsilon,n}_s\rangle_{\mathcal{F}_e^*}+\varphi(Y^{\varepsilon,n}_s)\nonumber\\
\leq&&\!\!\!\!\!\!\!\!\varphi(X_s)+\varepsilon c\big(|Y^{\varepsilon,n}_s|_{2}^{2}+\mu(E)\big),~~ds\otimes\Bbb{P}-a.e.,
\end{eqnarray*}
and
\begin{eqnarray*}
&&\!\!\!\!\!\!\!\!\Bbb{E}\|X_t-Y^{\varepsilon,n}_t\|^2_{\mathcal{F}_e^*}+2\Bbb{E}\int_0^t\varphi(X_s)ds\nonumber\\
\leq&&\!\!\!\!\!\!\!\!\Bbb{E}\|x_0-y_0^n\|^2_{\mathcal{F}_e^*}+2\Bbb{E}\int_0^t\varphi(X_s)ds+2\Bbb{E}\int_0^t\varepsilon c\big(|Y^{\varepsilon,n}_s|_{2}^{2}+\mu(E)\big)ds\nonumber\\
&&\!\!\!\!\!\!\!\!-2\Bbb{E}\int_0^t\langle \varepsilon LY^{\varepsilon,n}_s,X_s-Y^{\varepsilon,n}_s\rangle_{\mathcal{F}_e^*}ds+C\Bbb{E}\int_0^t\|X_s-Y^{\varepsilon,n}_s\|^2_{\mathcal{F}_e^*}ds,\ \forall t\in[0,T].
\end{eqnarray*}
By \eqref{rrw1}, the fact that $F_{1,2}\subset \mathcal{F}_e$ continuously and densely, and by \eqref{3.4} and Young's inequality,
\begin{eqnarray*}
&&\!\!\!\!\!\!\!\!\langle \varepsilon LY^{\varepsilon,n}_s,X_s-Y^{\varepsilon,n}_s\rangle_{\mathcal{F}_e^*}\nonumber\\
\leq&&\!\!\!\!\!\!\!\!\varepsilon\|LY^{\varepsilon,n}_s\|_{\mathcal{F}_e^*}\cdot\|X_s-Y^{\varepsilon,n}_s\|_{\mathcal{F}_e^*}\nonumber\\
=&&\!\!\!\!\!\!\!\!\varepsilon\|Y^{\varepsilon,n}_s\|_{\mathcal{F}_e}\cdot\|X_s-Y^{\varepsilon,n}_s\|_{\mathcal{F}_e^*}\nonumber\\
\leq&&\!\!\!\!\!\!\!\!\varepsilon^{\frac{4}{3}}\|Y^{\varepsilon,n}_s\|^2_{F_{1,2}}+\varepsilon^{\frac{2}{3}}\|X_s-Y^{\varepsilon,n}_s\|^2_{\mathcal{F}_e^*},~~ds\otimes\Bbb{P}-a.e.,
\end{eqnarray*}
hence, $\forall t\in [0,T]$,
\begin{eqnarray*}
&&\!\!\!\!\!\!\!\!\Bbb{E}\|X_t-Y^{\varepsilon,n}_t\|^2_{\mathcal{F}_e^*}\nonumber\\
\leq&&\!\!\!\!\!\!\!\!\Bbb{E}\|x_0-y_0^n\|^2_{\mathcal{F}_e^*}+2\Bbb{E}\int_0^t\varepsilon C\big(|Y^{\varepsilon,n}_s|_{2}^{2}+\mu(E)\big)ds\nonumber\\
&&\!\!\!\!\!\!\!\!+\Bbb{E}\int_0^t\big(\varepsilon^{\frac{4}{3}}\|Y^{\varepsilon,n}_s\|^2_{F_{1,2}}+\varepsilon^{\frac{2}{3}}\|X_s-Y^{\varepsilon,n}_s\|^2_{\mathcal{F}_e^*}\big)ds+C\Bbb{E}\int_0^t\|X_s-Y^{\varepsilon,n}_s\|^2_{\mathcal{F}_e^*}ds.
\end{eqnarray*}
Taking $\varepsilon\rightarrow0$, then $n\rightarrow\infty$, we get that $\forall t\in [0,T]$,
\begin{eqnarray*}
&&\!\!\!\!\!\!\!\!\Bbb{E}\|X_t-Y_t\|^2_{\mathcal{F}_e^*}\leq\Bbb{E}\|x_0-y_0\|^2_{\mathcal{F}_e^*}+C\Bbb{E}\int_0^t\|X_s-Y_s\|^2_{\mathcal{F}_e^*}ds,
\end{eqnarray*}
where $Y$ is the SVI solution which has been constructed from $\{Y^{\varepsilon,n}\}$ in the limiting procedure. By Gronwall's inequality, one has
\begin{eqnarray*}
\Bbb{E}\|X_t-Y_t\|^2_{\mathcal{F}_e^*} \leq e^{Ct}\Bbb{E}\|x_0-y_0\|^2_{\mathcal{F}_e^*},~~\forall t\in[0,T].
\end{eqnarray*}
Clearly, the above inequality implies \eqref{th1}. This completes the proof of Theorem \ref{theorem}.
\end{proof}

\section{Applications}\label{applications}

In this section, we give some examples of transient Dirichlet forms on finite measure spaces with their associated semigroups satisfying \eqref{uniqueness assumption}.

Let $E:=\mathcal{O}$, which is a bounded domain of $\Bbb{R}^d$, $d\geq3$, with smooth boundary $\partial\mathcal{O}$. Let $d\mu:=dx$ be the Lebesgue measure. Let $H_0^1$ denote the Sobolev space of $L^2(\mathcal{O})$ functions whose first order weak derivatives exist and belong to $L^2(\mathcal{O})$, and have zero trace. The norm on $H_0^1$ is defined as $\|u\|_{H_0^1}=|\nabla u|_2$, $\forall u\in H_0^1$.

%

\begin{example}\label{example for laplace}
Consider the Dirichlet form $(\mathcal{E}, D(\mathcal{E}))$ on $L^2(\mathcal{O})$, which is generated by the $L:=$Laplacian with zero boundary condition on the boundary $\partial\mathcal{O}$, let $\{P_t\}_{t\geq0}$ be the corresponding semigroup, then $D(\mathcal{E})=H_0^1$ and
$$ \mathcal{E}(u,u)=|\nabla u|^2_2,\ \forall u\in H_0^1.$$
From \cite[page 279, Theorem 3]{Evans}, we know that
\begin{eqnarray}\label{poincare O}
|u|_{\frac{2d}{d-2}}\leq C|\nabla u|_2,\ \forall u\in H_0^1,
\end{eqnarray}
where $C$ depends on $d$ and $\mathcal{O}$. This implies $(\mathcal{E}, H_0^1)$ is transient. From \cite[page 233]{SSV}, \eqref{poincare O} is equivalent to
\begin{eqnarray*}\label{ultrobounded}
\|P_t\|_{1\rightarrow\infty}\leq ct^{-\frac{d}{2}},\ \forall t>0,
\end{eqnarray*}
which by Remark \ref{ultraboundeded} implies \eqref{uniqueness assumption}.
\end{example}

%

\begin{example}\label{Bernstein function proposition 1}
Let $(\mathcal{E}, D(\mathcal{E}))$ and $\{P_t\}_{t\geq0}$ be defined as in Example \ref{example for laplace}. Let
\begin{eqnarray*}\label{Bernstein function1}
f(\lambda)=\lambda^\alpha,\ 0<\alpha<1,
\end{eqnarray*}
which is a Bernstein function (see e.g. \cite[page 21, Definition 3.1]{SSV}). Let $\{P^f_t\}_{t\geq0}$ be the subordinate semigroup of $-f(-\Delta(\mathcal{O}))=-(-\Delta(\mathcal{O}))^\alpha$, where $-\Delta(\mathcal{O})$ denotes $-\Delta$ on $\mathcal{O}$ with zero boundary condition. Then
\begin{eqnarray}\label{Ef sobolev inequality}
\|P^f_t\|_{1\rightarrow2}< \infty,\ \forall t>0.
\end{eqnarray}
\end{example}

\begin{proof}
The Dirichlet form $(\mathcal{E}^f,D(\mathcal{E}^f))$ on $L^2(\mathcal{O})$ associated with $-f(-\Delta(\mathcal{O}))=-(-\Delta(\mathcal{O}))^\alpha$ is transient (see e.g. \cite[page 257]{SSV}). By \cite[page 75, Theorem 2.4.2]{Davis}, \eqref{poincare O} is equivalent to
\begin{eqnarray*}
\|P_t\|_{2\rightarrow\infty}\leq c{t^{-\frac{d}{4}}},\ \forall t>0.
\end{eqnarray*}
Note that

\begin{eqnarray*}
c{t^{-\frac{d}{4}}}\leq \exp\{c_{p,d}t^{-\frac{1}{p-1}}\},\ \forall t>0, p>1, d\geq3,
\end{eqnarray*}
where $c_{p,d}$ is a constant depends on $d$ and $p$. Let $p>\frac{1}{\alpha}>1$, then
\begin{eqnarray}\label{integrable}
\int_1^\infty \frac{1}{f(\lambda^p)}d\lambda=\int_1^\infty \frac{1}{\lambda^{p\alpha}}d\lambda<\infty.
\end{eqnarray}
Hence, by \cite[page 241, Corollary 13.45.(i)]{SSV} we know that $\{P^f_t\}_{t\geq0}$ is ultrabounded, i.e. $\|P^f_t\|_{2\rightarrow\infty}< \infty,\ \forall t>0$, since $\{P_t^f\}_{t\geq0}$ is symmetric, we then have \eqref{Ef sobolev inequality}, which corresponds to \eqref{uniqueness assumption}.

\end{proof}

%

\begin{example}\label{Bernstein function proposition 2}
Let $(\mathcal{E}, D(\mathcal{E}))$ and $\{P_t\}_{t\geq0}$  be defined as in Example \ref{example for laplace}. Let $f$ be a Bernstein function as following
\begin{eqnarray*}\label{Bernstein function2}
f(\lambda)=(\lambda+1)^\alpha-1,\ 0<\alpha<1.
\end{eqnarray*}
Let $\{P^f_t\}_{t\geq0}$ be the subordinate semigroup of $-f(-\Delta(\mathcal{O}))=-(-\Delta(\mathcal{O})+I(\mathcal{O}))^\alpha+I(\mathcal{O})$, where $I$ is the identity function. Then
\begin{eqnarray*}
\|P^f_t\|_{1\rightarrow2}< \infty,\ \forall t>0.
\end{eqnarray*}
\end{example}
\begin{proof}
The proof is the same as the proof of Example \ref{Bernstein function proposition 1}, if we can prove that \eqref{integrable} holds for $f(\lambda)=(\lambda+1)^\alpha-1$, $0<\alpha<1$. Let $p>\frac{1}{\alpha}>1$, then
\begin{eqnarray*}
\int_1^\infty \frac{1}{f(\lambda^p)}d\lambda=\int_1^\infty\frac{1}{(\lambda^p+1)^\alpha-1}d\lambda<\int_1^\infty \frac{1}{\lambda^{p\alpha}}d\lambda<\infty.
\end{eqnarray*}
\end{proof}

Let $\{\nu_t\}_{t>0}$ on $\Bbb{R}^d$ be a continuous symmetric convolution semigroup, its L\'{e}vy-Khinchin formula (\cite{I}) is as follows:
\begin{eqnarray*}
           \left\{
           \begin{array}{ll}
             &\!\!\!\!\!\!\!\!\hat{\nu}_t(x)\big(=\int_{\Bbb{R}^d}e^{i(x,y)}\nu_t(dy)\big)=e^{-t\phi(x)},\\
             &~\\
             &\!\!\!\!\!\!\!\!\psi(x)=\frac{1}{2}(Sx,x)+\int_{\Bbb{R}^d}(1-\cos(x,y))J(dy),
           \end{array}
         \right.
\end{eqnarray*}
where
\begin{eqnarray*}
&&S~\text{is~a~non-negative~definite}~d\times d~\text{symmetric~matrix},\label{con2}\\
&&J~\text{is~a~symmetric~measure~on}~\Bbb{R}^d\setminus\{0\}~\text{such~that}~\int_{\Bbb{R}^d\setminus\{0\}}\frac{|x|^2}{1+|x|^2}J(dx)<\infty.\label{con3}
\end{eqnarray*}
$\{\nu_t\}_{t>0}$ defines a Markovian transition function $p_t(x,A)$ on $(\Bbb{R}^d,\mathcal{B}(\Bbb{R}^d))$ by
\begin{eqnarray*}\label{transition}
p_t(x,A)=\nu_t(A-x).
\end{eqnarray*}
When $\nu_t(dy)$ is of the form $g_t(y)dy$ with $g_t\in L^2(\Bbb{R}^d)$, then
\begin{eqnarray*}
P_tu(x)=\int_{\Bbb{R}^d}u(x+y)g_t(y)dy=(g_t\ast u)(x),\ \ \forall u\in L^1(\Bbb{R}^d),
\end{eqnarray*}
which by Young's inequality (\cite[page 5, Lemma 1.4]{BCD}) implies
\begin{eqnarray*}
\|P_t\|_{1\rightarrow2}<\infty,\ \forall t>0.
\end{eqnarray*}
By \cite[page 31, Example 1.4.1]{FOT}, the regular Dirichlet form $\mathcal{E}$ on $L^2(\Bbb{R}^d)$ determined by the above $\{\nu_t\}_{t>0}$ is
\begin{eqnarray}\label{definition of DE}
           \left\{
           \begin{array}{ll}
             &\!\!\!\!\!\!\!\!\mathcal{E}(u,v)=\int_{\Bbb{R}^d}\hat{u}(x)\bar{\hat{v}}(x)\phi(x)dx,\\
             &~\\
             &\!\!\!\!\!\!\!\!D(\mathcal{E})=\{u\in L^2(\Bbb{R}^d):\int_{\Bbb{R}^d}|\hat{u}(x)|^2\phi(x)dx<\infty\}.
           \end{array}
         \right.
\end{eqnarray}
By \cite[page 48, Example 1.5.2]{FOT}, the extended Dirichlet space $(\mathcal{E},\mathcal{F}_e)$ of \eqref{definition of DE} is:
\begin{eqnarray*}\label{definition of Fe}
           \left\{
           \begin{array}{ll}
             &\!\!\!\!\!\!\!\!\mathcal{F}_e=\{u\in L^1_{loc}(\Bbb{R}^d):u\ \text{is\ a\ tempered\ distribution\ and}\ \hat{u}\in L^2(\phi\cdot dx)\},\\
             &~\\
             &\!\!\!\!\!\!\!\!\mathcal{E}(u,v)=\int_{\Bbb{R}^d}\hat{u}(x)\bar{\hat{v}}(x)\phi(x)dx, u,v\in \mathcal{F}_e.
           \end{array}
         \right.
\end{eqnarray*}
It is shown in \cite[page 48, Example 1.5.2]{FOT} that $(\mathcal{E},D(\mathcal{E}))$ is transient iff $\frac{1}{\phi}$ is locally integrable on $\Bbb{R}^d$. Typical examples (cf. \cite[Example 5.31]{BBKRSV} for more examples of $\phi$ and their applications) are:

%

\begin{example}\label{transient}
 (i)~~$\phi(x)=|x|^\alpha$, $(\mathcal{E},D(\mathcal{E}))$ is transient if $0<\alpha\leq2$, $\alpha<d$. The generator of $(\mathcal{E},D(\mathcal{E}))$ is the fractional Laplacian $\Delta^{\frac{\alpha}{2}}$. In this case, the extended Dirichelt space $(\mathcal{E},\mathcal{F}_e)$ can be identified with the Riesz potential space of index $\frac{\alpha}{2}$ (\cite[page 50]{FOT}). For the case $\alpha=2$, the associated symmetric Hunt processes (\cite[Chapter 4]{FOT}) is the $d$-dimensional Brownian motion with variance being equal to $2t$; \\  (ii)~~$\phi(x)=log(1+|x|^2)$, $(\mathcal{E},D(\mathcal{E}))$ is transient if $d>2$, it is related to the subordinate Brownian motion (\cite[page 110]{BBKRSV}) known as the variance gamma process, which has applications in finance (\cite{GMY});\\
 (iii)~~$\phi(x)=(|x|^2+m^{\frac{\alpha}{2}})^{\frac{2}{\alpha}}-m$, $m>0$, $(\mathcal{E},D(\mathcal{E}))$ is transient if $0<\alpha<2$,  $d>2$, it is related to the subordinate Brownian motion called the symmetric relativistic $\alpha$-stable process;\\
 (iv)~~$\phi(x)= \log((1+|x|^2)+\sqrt{(1+|x|^2)^2-1})$, $(\mathcal{E},D(\mathcal{E}))$ is transient if $d>1$, it is related to the subordinate Brownian motion with the Bessel subordinator.
 \end{example}

Now, let us construct transient Dirichelt forms on $G$, which is an open set of $\Bbb{R}^d$ with finite measure. Set $L^2_G(\Bbb{R}^d)=\{u\in L^2(\Bbb{R}^d):u=0\ \text{a.e.}\ \text{on}\ D^c\}$, which can be identified with $L^2(G)$. The Dirichlet form $\mathcal{E}$ on $G$ is denoted by $\mathcal{E}_G$ with $D(\mathcal{E}_G)=\{u\in D(\mathcal{E}): \tilde{u}=0\ \text{q.e.}\ \text{on}\ D^c\}$, where $\tilde{u}$ is a quasi continuous version of $u$. By \cite[page 174, Theorem 4.4.3]{FOT}, $(\mathcal{E}_G,D(\mathcal{E}_G))$ is a regular Dirichlet form on $L^2_G(\Bbb{R}^d)$. By \cite[page 175, Theorem 4.4.4]{FOT}, $(\mathcal{E}_G,D(\mathcal{E}_G))$ is transient if $(\mathcal{E},D(\mathcal{E}))$ is transient.

By \cite[page 173, Theorem 4.4.2]{FOT}, the strongly continuous semigroup $\{P_t^G\}_{t>0}$ on $L^2_G(\Bbb{R}^d)$ corresponding to $(\mathcal{E}_G,D(\mathcal{E}_G))$ is determined by the restriction of transient function $\{p_t\}_{t>0}$ to $(G,\mathcal{B}(G))$ defined by \cite[page 153, (4.1.2)]{FOT} for $X=G$.
It remains to prove that
\begin{eqnarray*}
\|P^G_t\|_{1\rightarrow2}<\infty,\ \forall t>0.
\end{eqnarray*}
Indeed, let $u\in L^1_G(\Bbb{R}^d):=\{u\in L^1(\Bbb{R}^d):u=0\ \text{a.e.}\ \text{on}\ G^c\}$ (which can be identified with $L^1(G)$), then
\begin{eqnarray*}
&&\!\!\!\!\!\!\!\!\Big(\int_{\Bbb{R}^d}|P_t^Gu(x)|^2dx\Big)^{\frac{1}{2}} \nonumber\\
\leq&&\!\!\!\!\!\!\!\!\Big(\int_{\Bbb{R}^d}\big(P_t^G|u|(x)\big)^2dx\Big)^{\frac{1}{2}} \nonumber\\
\leq&&\!\!\!\!\!\!\!\!\Big(\int_{\Bbb{R}^d}\big(P_t|u|(x)\big)^2dx\Big)^{\frac{1}{2}} \nonumber\\
\leq&&\!\!\!\!\!\!\!\!C\int_{\Bbb{R}^d}|u(x)|dx.
\end{eqnarray*}

\begin{remark}\label{example of uniqueness}
There are interesting examples satisfying \textbf{(T)} and \eqref{uniqueness assumption} on manifolds (see e.g. \cite[Section 4]{Varo}, \cite{AL}, \cite{VSC}) or graphs (see e.g. \cite[page:3, Example 1.5]{HKSW}).
\end{remark}

\section{Appendix}\label{Appendix}

\subsection{Moreau-Yosida approximation of maximal monotone operators}\label{Moreau Yosida}

\vspace{2mm}

Let $\psi:\Bbb{R}\rightarrow [0,\infty)$ be a convex, lower semicontinuous function with $\psi(0)=0$. Let $\beta (:=\partial\psi): \mathbb{R}\rightarrow\!2^\mathbb{R}$ be the subdifferential of $\psi$. We recall some properties of $\beta^\varepsilon$, which is the Yosida approximation of $\beta$, i.e.,
\begin{eqnarray*}\label{definition yosida}
\beta^\varepsilon(r)=\frac{1}{\varepsilon}(r-(1+\varepsilon\beta)^{-1}(r))\in\beta((1+\varepsilon\beta)^{-1}r),~~\forall r\in\Bbb{R},
\end{eqnarray*}
and $\psi^\varepsilon$, which is the Moreau regularization of $\psi$, i.e.,
\begin{eqnarray*}
\psi^\varepsilon(r):=\inf\Big\{\frac{|r-\overline{r}|^2}{2\varepsilon}+\psi(\overline{r}); \overline{r}\in\Bbb{R}\Big\}, ~~\forall r\in\Bbb{R}.
\end{eqnarray*}

\begin{lemma}\label{lemma3.1}
For $\varepsilon>0$, $\psi^\varepsilon$ is convex, continuous, Gateaux differentiable and $\beta^\varepsilon=\partial \psi^\varepsilon$. $\beta^\varepsilon$ is monotone, Lipschitz continuous with Lipschitz constant $\frac{1}{\varepsilon}$.
\begin{eqnarray}\label{yosida1}
|\beta^\varepsilon(r)|\leq|\beta(r)|,~~ \forall r\in \Bbb{R},
\end{eqnarray}
where
$|\beta(r)|:=\inf\{|\eta|:\eta\in \beta(r)\},\forall r\in\Bbb{R}$.
\begin{eqnarray}\label{yosida2}
\psi(J^\varepsilon(r))\leq\psi^\varepsilon(r)\leq\psi(r),~~ \forall r\in \Bbb{R},
\end{eqnarray}
where for every $r\in\Bbb{R}$, $s:=J^\varepsilon(r)$ is the solution to the equation
$$0\in s-r+\varepsilon\beta(s).$$
Furthermore,
\begin{eqnarray}\label{yosida3}
|\psi(r)-\psi^\varepsilon(r)|\leq\varepsilon|\beta(r)|^2,~~ \forall r\in \Bbb{R}.
\end{eqnarray}
For all $r,r'\in\Bbb{R}$, $\varepsilon_1, \varepsilon_2\in(0,\infty)$,
\begin{eqnarray}\label{yosida4}
\big(\beta^{\varepsilon_1}(r)-\beta^{\varepsilon_2}(r')\big)(r-r')\geq -\frac{1}{2}(\varepsilon_1+\varepsilon_2)\big(|\beta^{\varepsilon_1}(r)|^2+|\beta^{\varepsilon_2}(r')|^2\big),
\end{eqnarray}
and if \eqref{beta} is satisfied, then
\begin{eqnarray}\label{beta3}
\big(\beta^{\varepsilon_1}(r)-\beta^{\varepsilon_2}(r')\big)(r-r')\geq-C(\varepsilon_1+\varepsilon_2)\big(|r|^2+|r'|^2+1\big).
\end{eqnarray}
\end{lemma}

\begin{proof}
The first statement and \eqref{yosida2} can be found in \cite[page 48, Theorem 2.9]{Barbu}. The second statement can be found in \cite[page 38, Proposition 2.2]{Barbu} and \cite[page 41, Proposition 2.3 (ii)]{Barbu}. \eqref{yosida1}, \eqref{yosida3} and \eqref{yosida4} can be found in \cite[page 98, Lemma 3.D.3]{Marus}, \cite[page 98, Lemma 3.D.6.]{Marus} and \cite[page 99, Lemma 3.D.8.]{Marus}, respectively. By \eqref{yosida4} and \eqref{beta} we can get \eqref{beta3}.

\end{proof}

\subsection{Lower semicontinuity of the energy potential funtional on $\mathcal{F}_e^*$}\label{Lower semicontinuity of varphi}

\vspace{2mm}
Since the measure is finite, it follows from \cite[page 425, Theorem A.4.1 (i)]{FOT} that, $F_{1,2}\cap L^\infty(\mu)$ is dense in $(F_{1,2},\|\cdot\|_{F_{1,2}})$. Since $(\mathcal{E},D(\mathcal{E}))$ is a transient Dirichlet space, it follows that $F_{1,2}\cap L^\infty(\mu)$ is dense in $(\mathcal{F}_e,\|\cdot\|_{\mathcal{F}_e})$.

Define $L^1(\mu)\cap\mathcal{F}_e^*$ in the following sense:
\begin{eqnarray*}
L^1(\mu)\cap\mathcal{F}_e^*:=\{v\in L^1(\mu)|\exists C\in(0,\infty)~\text{such~that}~\mu(uv)\leq C\|u\|_{\mathcal{F}_e},~\forall u\in F_{1,2}\cap L^\infty(\mu)\}.
\end{eqnarray*}
For each $v\in L^1(\mu)\cap\mathcal{F}_e^*$, one can define a map
\begin{eqnarray*}
\overline{v}(u):=\mu(uv),\ \forall u\in F_{1,2}\cap L^\infty(\mu),
\end{eqnarray*}
by continuity, it can be uniquely extended to a bounded linear functional on $\mathcal{F}_e$. Then $L^1(\mu)\cap\mathcal{F}_e^*$ can be considered as a subset of $\mathcal{F}_e^*$ by identifying $v$ with $\overline{v}$, and we have
\begin{eqnarray}\label{FeL1}
_{\mathcal{F}_e^*}\langle v,u\rangle_{\mathcal{F}_e}=\mu(uv),\ \forall u\in F_{1,2}\cap L^\infty(\mu), v\in L^1(\mu)\cap\mathcal{F}_e^*.
\end{eqnarray}
Similar to Lemma \ref{RRW} (ii) (cf. \cite[page 134]{RRW} for its proof ), we also have
\begin{eqnarray*}\label{FeL2}
\langle \overline{L}u, \overline{v}\rangle_{\mathcal{F}_e^*}=-\mu(uv),\ \forall u\in F_{1,2}\cap L^\infty(\mu), v\in L^1(\mu)\cap\mathcal{F}_e^*.
\end{eqnarray*}

Recall from \eqref{convex} that $\varphi:\mathcal{F}_e^*\rightarrow [0,\infty]$ is defined as following
\begin{eqnarray*}
\varphi(v):=\left\{
           \begin{array}{ll}
             &\!\!\!\!\!\!\!\! \int_E \psi(v(x))\mu(dx),~~~~~ v\in L^1(\mu)\cap\mathcal{F}_e^*,\\
             &\!\!\!\!\!\!\!\!+\infty,~~~~~~~~~~~~~~~~~ ~~~~\text{else}.
           \end{array}
         \right.
\end{eqnarray*}
\begin{proposition}\label{cvx}
Suppose that \textbf{(H1)} and \textbf{(H2)} are satisfied. Then $\varphi$ is a proper, convex, lower-semicontinuous function on $\mathcal{F}_e^*$.
\end{proposition}

\begin{proof}
We first claim that $\widetilde{\varphi}:L^1(\mu)\rightarrow [0,\infty]$, which is defined as following
\begin{eqnarray*}
\widetilde{\varphi}(v):=\left\{
           \begin{array}{ll}
             &\!\!\!\!\!\!\!\! \int_E \psi(v(x))\mu(dx),~~~~~ v\in L^1(\mu)\cap\mathcal{F}_e^*,\\
             &\!\!\!\!\!\!\!\!+\infty,~~~~~~~~~~~~~~~~~ ~~~~\text{else},
           \end{array}
         \right.
\end{eqnarray*}
is a proper, convex, lower-semicontinuous function on $L^1(\mu)$.

Note that $\psi(0)=0$, hence $\widetilde{\varphi}$ is a proper function on $L^1(\mu)$ and $\varphi$ is a proper function on $\mathcal{F}_e^*$. Since $\psi:\Bbb{R}\rightarrow[0,\infty)$ is convex, it implies that $\widetilde{\varphi}$ is a convex function on $L^1(\mu)$ and $\varphi$ is a convex function on $\mathcal{F}_e^*$.

To prove $\widetilde{\varphi}$ is lower-semicontinuous on $L^1(\mu)$, we have to prove that for every $\xi\in\Bbb{R}$, the level subset $\{v\in L^1(\mu);\widetilde{\varphi}(v)\leq\xi\}$ is closed in $L^1(\mu)$. Consider a sequence $\{v^n\}_{n\in\Bbb{N}}\subset L^1(\mu)$ such that $\widetilde{\varphi}(v^n)\leq\xi$ (which implies that $\{v^n\}_{n\in\Bbb{N}}\subset L^1(\mu)\cap\mathcal{F}_e^*$), and with an element $\widetilde{v}\in L^1(\mu)$ such that $v^n\rightarrow \widetilde{v}$ in $L^1(\mu)$. We must prove that $\widetilde{\varphi}(\widetilde{v})\leq\xi$. Since $v^n\rightarrow \widetilde{v}$ in $L^1(\mu)$, there exists a subsequence $\{v^{n_k}\}_{k\in\Bbb{N}}$ such that $v^{n_k}\rightarrow \widetilde{v}$, a.s.. Since $\psi$ is lower-semicontinuous,
$$\psi(\widetilde{v}(x))\leq\liminf_{k\rightarrow\infty}\psi(v^{n_k}(x)),\ a.s..$$
By integrating the above inequality over $E$ with respect to $\mu$, and by Fatou's lemma, we obtain that
\begin{eqnarray*}
\int_E\psi(\widetilde{v}(x))\mu(dx)&&\!\!\!\!\!\!\!\!\leq \int_E\liminf_{k\rightarrow\infty}\psi (v^{n_k}(x))\mu(dx)\nonumber\\
&&\!\!\!\!\!\!\!\!\leq\liminf_{k\rightarrow\infty}\int_E\psi (v^{n_k}(x))\mu(dx)\nonumber\\
&&\!\!\!\!\!\!\!\!\leq\xi.
\end{eqnarray*}

To prove $\varphi$ is lower-semicontinuous on $\mathcal{F}_e^*$, we have to prove that for every $\xi\in\Bbb{R}$, the level subset $\{v\in \mathcal{F}_e^*;\varphi(v)\leq\xi\}$ is closed in $\mathcal{F}_e^*$. Consider a sequence $\{v_n\}_{n\in\Bbb{N}}\subset \mathcal{F}_e^*$ such that $\varphi(v_n)\leq\xi$ (which implies that $\{v_n\}_{n\in\Bbb{N}}\subset L^1(\mu)\cap\mathcal{F}_e^*$), and with an element $v\in \mathcal{F}_e^*$ such that $v_n\rightarrow v$ in $\mathcal{F}_e^*$. We must prove that $\varphi(v)\leq\xi$. But since $\widetilde{\varphi}$ is convex and lower-semicontinuous on $L^1(\mu)$, it suffices to show that $\{v_n\}_{n\in\Bbb{N}}$ is weakly sequentially compact in $L^1(\mu)$ (\cite[page 80, Theorem 2.3]{Mali}). Since the measure is finite, according to Dunford-Pettis (\cite[page 175, Theorem 2.54]{FL}), we need to prove the following two things:\\
(i)~$\{v_n\}_{n\in\Bbb{N}}$ is bounded in $L^1(\mu)$;\\
(ii)~for every $\varepsilon>0$, there exists $\delta>0$ such that
$$\int_{E'}|v_n(x)|\mu(dx)\leq\varepsilon$$
for every measurable set $E'\subset E$ with $\mu(E')\leq\delta$.\\

By \eqref{psi}, we can find a $C>0$ such that $\frac{\psi(r)}{|r|}\geq 1$ if $|r|\geq C$. Then,
\begin{eqnarray*}
&&\!\!\!\!\!\!\!\!\int_E|v_n(x)|\mu(dx)\nonumber\\
=&&\!\!\!\!\!\!\!\!\int_{\{|v_n|\geq C\}}|v_n(x)|\mu(dx)+\int_{\{|v_n|<C)\}}|v_n(x)|\mu(dx)\nonumber\\
\leq&&\!\!\!\!\!\!\!\!\int_{\{|v_n|\geq C\}}\psi(v_n(x))\mu(dx)+\int_{\{|v_n|<C\}}C\mu(dx)\nonumber\\
\leq&&\!\!\!\!\!\!\!\!\xi+C\mu(E),
\end{eqnarray*}
which implies (i).

To prove (ii), by the proof of \cite[page 267, 4.5.3. Proposition]{Bogachev}, we only need to show that $\{v_n\}_{n\in\Bbb{N}}$ is uniformly integrable (\cite[page 267, 4.5.1. Definition]{Bogachev}), which is true according to the proof of \cite[page 272, 4.5.9. Theorem]{Bogachev} since \eqref{psi} and $\varphi(v_n)\leq\xi$ are fulfilled.

Therefore, there exists a subsequence $\{v_{n_k}\}_{k\in\Bbb{N}}$ and $v'\in L^1(\mu)$ such that $v_{n_k}\rightharpoonup v'$ in $L^1(\mu)$ as $k\rightarrow\infty$. So, for all $u\in F_{1,2}\cap L^\infty(\mu)$, by \eqref{FeL1} applied to $v_{n_k}$ replacing $v$, we have
\begin{eqnarray*}
_{\mathcal{F}_e^*}\langle v,u\rangle_{\mathcal{F}_e}=\lim_{k\rightarrow\infty}~_{\mathcal{F}_e^*}\langle v_{n_k},u\rangle_{\mathcal{F}_e}=\lim_{k\rightarrow\infty}\int_E v_{n_k}ud\mu=\int_E v'ud\mu,
\end{eqnarray*}
which implies that $v'\in L^1(\mu)\cap\mathcal{F}_e^*$ and that $v=v'$ because $F_{1,2}\cap L^\infty(\mu)$ is dense in $\mathcal{F}_e$. Consequently,
$$\varphi(v)=\widetilde{\varphi}(v)\leq\liminf_{k\rightarrow\infty}\widetilde{\varphi}(v_n)\leq\xi.$$
\end{proof}

\begin{proposition}\label{construction of un}
Suppose that \textbf{(H1)}, \textbf{(H2)} and \eqref{uniqueness assumption} are satisfied. Then, for any $v$ satisfying $\varphi(v)<\infty$, there exists a sequence $\{v_n\}_{n\in\Bbb{N}}$ with $v_n\in L^2(\mu)\cap\mathcal{F}_e^*$ such that $v_n\rightharpoonup v$ in $\mathcal{F}_e^*$ and $\varphi(v)=\lim_{n\rightarrow\infty}\varphi(v_n)$.
\end{proposition}

\begin{proof}
Since \eqref{uniqueness assumption} is satisfied, by \cite[page 156, Theorem 3.3.15]{Wang}, we know that the following inequality holds:
\begin{eqnarray*}
\mu(u^2)\leq r\mathcal{E}(u,u)+b(r)\mu(|u|)^2,\ \forall r>0, u\in F_{1,2},
\end{eqnarray*}
where
\begin{eqnarray*}
b(r)=\inf_{s\leq r,t>0}\frac{s\|P_t\|_{1\rightarrow\infty}}{t}\exp\{t/s-1\}.
\end{eqnarray*}
Since $(\mathcal{E},F_{1,2})$ is transient, there exists a bounded $\mu$-integrable function $g$ strictly positive $\mu$-a.e. such that
\begin{eqnarray}\label{Ef transient}
\int_E|u|gd\mu\leq \sqrt{\mathcal{E}(u,u)}, \ \forall u\in F_{1,2}.
\end{eqnarray}
From \cite[page 142, Theorem 3.2.1]{Wang}, we also know that for the above $g$, there exists $\widetilde{b}:(0,\infty)\rightarrow(0,\infty)$ such that
\begin{eqnarray}\label{Ef g}
\mu(u^2)\leq r \mathcal{E}(u,u)+\widetilde{b}(r)\Big(\int_E|u|gd\mu\Big)^2,\ \forall r>0, u\in F_{1,2}.
\end{eqnarray}
Taking \eqref{Ef transient} into \eqref{Ef g}, we have
\begin{eqnarray*}
\mu(u^2)\leq r \mathcal{E}(u,u)+\widetilde{b}(r)\mathcal{E}(u,u),\ \forall r>0, u\in F_{1,2}.
\end{eqnarray*}
Therefore, the Poincar\'{e} inequality holds, i.e. there exists a constant $c>0$ such that
\begin{eqnarray*}
\mu(u^2)\leq c\mathcal{E}(u,u),\ \forall u\in F_{1,2}.
\end{eqnarray*}
Since $\mathcal{F}_e$ is the completion of $F_{1,2}$ with respect to the norm $\|\cdot\|_{\mathcal{F}_e}=\mathcal{E}(\cdot,\cdot)^{\frac{1}{2}}$, it follows that $F_{1,2}=\mathcal{F}_e$, $F^*_{1,2}=\mathcal{F}^*_e$ and $L^2(\mu)\cap \mathcal{F}_e^*=L^2(\mu)\cap F^*_{1,2}=L^2(\mu)$.

From \cite[page 32, Lemma 1.3.3]{FOT}, we know that $\forall u\in F_{1,2}$, one has $P_tu\in F_{1,2}$ and
\begin{eqnarray}\label{Ptu}
\lim_{t\rightarrow0}P_tu=u.
\end{eqnarray}
Define $P'_t:F^*_{1,2}\rightarrow F^*_{1,2}$ by
\begin{eqnarray}\label{definition of P't}
_{F^*_{1,2}}\langle P'_tv,u\rangle_{F_{1,2}}:=~_{F^*_{1,2}}\langle v,P_tu\rangle_{F_{1,2}},\ \forall u\in F_{1,2}, v\in F^*_{1,2}.
\end{eqnarray}
Then, it follows from \eqref{Ptu} that
\begin{eqnarray*}
\lim_{t\rightarrow0}~_{F^*_{1,2}}\langle P'_tv,u\rangle_{F_{1,2}}=\lim_{t\rightarrow0}~_{F^*_{1,2}}\langle v,P_tu\rangle_{F_{1,2}}=~_{F^*_{1,2}}\langle v,u\rangle_{F_{1,2}},\ \forall u\in F_{1,2}, v\in F^*_{1,2}.
\end{eqnarray*}
Hence,
\begin{eqnarray}\label{Pt'}
 P'_tv\rightharpoonup v\ \ \text{in}\ F^*_{1,2}\ \ \text{as}\ t\rightarrow0,\ \ \forall v\in F^*_{1,2}.
\end{eqnarray}

Now let $v\in L^2(\mu)$. Recall a Gelfand triple $F_{1,2}\subset L^2(\mu)(\equiv (L^2(\mu))^*)\subset F^*_{1,2}$, then by \eqref{definition of P't}, \cite[page 69, (4.2)]{LR}, $P_t$ is symmetric on $L^2(\mu)$, we have
\begin{eqnarray*}
_{F^*_{1,2}}\langle P'_tv,u\rangle_{F_{1,2}}=~_{F^*_{1,2}}\langle v,P_tu\rangle_{F_{1,2}}=\langle v,P_tu\rangle_2=\langle P_tv,u\rangle_2=~_{F^*_{1,2}}\langle P_tv,u\rangle_{F_{1,2}},\ \forall u\in F_{1,2}.
\end{eqnarray*}
Hence for $v\in L^2(\mu)$, $P_tv=P'_tv$ in $F^*_{1,2}$.

By \eqref{convex}, $\varphi(v)<\infty$ implies that $v\in L^1(\mu)\cap F^*_{1,2}$. Since \eqref{uniqueness assumption} is satisfied, then $\forall v\in L^1(\mu)\cap F^*_{1,2}$, we have $P_tv\in L^2(\mu)\subset L^1(\mu)\cap F^*_{1,2}$ and
\begin{eqnarray}\label{PtPt'}
P_tv=P'_tv \ \text{in}\ F^*_{1,2}.
\end{eqnarray}
Indeed, since $L^\infty(\mu)$ is dense in $L^1(\mu)$, for any $v\in L^1(\mu)\cap F^*_{1,2}$, there exists a sequence $\{v^n\}_{n\in\Bbb{N}}\subset L^\infty(\mu)$ such that $v^n\rightarrow v$ in $L^1(\mu)$ as $n\rightarrow \infty$. Then, by \eqref{FeL1}, $P_t$ is symmetric on $L^2(\mu)$, $P_t$ is a contraction semigroup on $L^\infty(\mu)$ (\cite[page 22, Theorem 1.4.1]{Davis}), and \eqref{definition of P't}, we know that $\forall u\in F_{1,2}\cap L^\infty(\mu)$, $v\in L^1(\mu)\cap F^*_{1,2}$,
\begin{eqnarray*}
&&\!\!\!\!\!\!\!\!_{F^*_{1,2}}\langle P_tv,u\rangle_{F_{1,2}}\nonumber\\
=&&\!\!\!\!\!\!\!\!\int_EP_tv\cdot ud\mu\nonumber\\
=&&\!\!\!\!\!\!\!\!\lim_{n\rightarrow\infty}\int_EP_tv^n\cdot ud\mu\nonumber\\
=&&\!\!\!\!\!\!\!\!\lim_{n\rightarrow\infty}\int_Ev^nP_t ud\mu\nonumber\\
=&&\!\!\!\!\!\!\!\!\int_EvP_tud\mu\nonumber\\
=&&\!\!\!\!\!\!\!\!~_{F^*_{1,2}}\langle v,P_tu\rangle_{F_{1,2}}\nonumber\\
=&&\!\!\!\!\!\!\!\!~_{F^*_{1,2}}\langle P'_tv,u\rangle_{F_{1,2}},
\end{eqnarray*}
since $F_{1,2}\cap L^\infty(\mu)$ is dense in $F_{1,2}$, \eqref{PtPt'} follows. Together with \eqref{Pt'}, we get
\begin{eqnarray*}
 P_tv\rightharpoonup v\ \ \text{in}\ F^*_{1,2}\ \ \text{as}\ t\rightarrow0,\ \ \forall v\in L^1(\mu)\cap F^*_{1,2},
\end{eqnarray*}
which also means
\begin{eqnarray}\label{P1n}
 P_{\frac{1}{n}}v\rightharpoonup v\ \ \text{in}\ F^*_{1,2}\ \ \text{as}\ n\rightarrow\infty,\ \ \forall v\in L^1(\mu)\cap F^*_{1,2}.
\end{eqnarray}

Since $\psi$ is convex, by \cite[page 83]{BV}, we know that
$$\psi(x)=\sup_{a,b\in\Bbb{R}}\{ax+b\ |\ az+b\leq \psi(z),\ \forall z\in\Bbb{R}\}.$$
Since $\psi(0)=0$, it follows that $b\leq 0$. Since $P_t$ is a sub-Markovian semigroup on $L^2(\mu)$, it follows that $b\leq P_tb$, $\forall t>0, b\leq0$. Set $v_n:=P_{\frac{1}{n}}v$, $n\in\Bbb{N}$, we then have

\begin{eqnarray}\label{vn}
\psi(v_n)&&\!\!\!\!\!\!\!\!=\sup_{a\in\Bbb{R},b\leq0}\Big\{aP_{\frac{1}{n}}v+b\ |\ az+b\leq \psi(z),\ \forall z\in\Bbb{R}\Big\}\nonumber\\
&&\!\!\!\!\!\!\!\!\leq\sup_{a\in\Bbb{R},b\leq0}\Big\{P_{\frac{1}{n}}(av+b)\ |\ az+b\leq \psi(z),\ \forall z\in\Bbb{R}\Big\}\nonumber\\
&&\!\!\!\!\!\!\!\!\leq P_{\frac{1}{n}}\Big(\sup_{a\in\Bbb{R},b\leq0}\Big\{av+b\ |\ az+b\leq \psi(z),\ \forall z\in\Bbb{R}\Big\}\Big)\nonumber\\
&&\!\!\!\!\!\!\!\!\leq P_{\frac{1}{n}}(\psi(v)).
\end{eqnarray}

By the weakly lower semicontinuity of $\varphi$ (\cite[page 157, Theorem 9.1]{BC}), \eqref{P1n}, \eqref{vn} and since $P_{\frac{1}{n}}$ is a contraction semigroup on $L^1(\mu)$ (\cite[page 22, Theorem 1.4.1]{Davis}), we have
\begin{eqnarray*}
\varphi(v)&&\!\!\!\!\!\!\!\!\leq \liminf_{n\rightarrow\infty}\varphi(v_n)\nonumber\\
&&\!\!\!\!\!\!\!\!\leq\limsup_{n\rightarrow\infty}\int_E\psi(v(x))\mu(dx)\nonumber\\
&&\!\!\!\!\!\!\!\!\leq\limsup_{n\rightarrow\infty}\int_EP_{\frac{1}{n}}\psi(v(x))\mu(dx)\nonumber\\
&&\!\!\!\!\!\!\!\!\leq\limsup_{n\rightarrow\infty}\int_E\psi(v(x))\mu(dx)\nonumber\\
&&\!\!\!\!\!\!\!\!=\varphi(v),
\end{eqnarray*}
which implies $\varphi(v)=\lim_{n\rightarrow\infty}\varphi(v_n)$.
\end{proof}

\subsection{Comparisons of different notions of solutions}\label{Comparison}

\vspace{2mm}


Let $\mathcal{F}_e^*\cap L^{p}(\mu)$, $1<p<\infty$, be defined as in \eqref{Fe*Lp}. Define
\begin{eqnarray*}
Av=\{-\bar{L}u,~u\in\mathcal{F}_e\cap L^{\frac{p}{p-1}}(\mu),~u(x)\in\beta(v(x)),~\mu-a.e.,~x\in E\},
\end{eqnarray*}
\begin{eqnarray*}
D(A)=\{v\in\mathcal{F}_e^*\cap L^{p}(\mu),~\varphi(v)<\infty,~\exists u\in\mathcal{F}_e\cap L^{\frac{p}{p-1}}(\mu), u(x)\in\beta(v(x)),~\mu-a.e.,~x\in E\}.
\end{eqnarray*}

Let $v\in D(A)$ and $-\bar{L}u\in Av$. Recall $\beta=\partial\psi$ and that $\varphi$ is convex by Proposition \ref{cvx}. Hence by \eqref{rrw3} and the convexity of $\psi$, we have for all $v'\in\mathcal{F}_e^*\cap L^{p}(\mu)$,
\begin{eqnarray*}
&&\!\!\!\!\!\!\!\!\langle-\bar{L}u,v-v'\rangle_{\mathcal{F}_e^*}\nonumber\\
=&&\!\!\!\!\!\!\!\!\int_Eu(x)\big(v(x)-v'(x)\big)\mu(dx)\nonumber\\
\geq&&\!\!\!\!\!\!\!\!\int_E\psi(v(x))-\psi(v'(x))\mu(dx)\nonumber\\
=&&\!\!\!\!\!\!\!\!\varphi(v)-\varphi(v').
\end{eqnarray*}
Consequently, $Av\subset\partial\varphi(v)$. Hence, \eqref{eq:1} implies that $X$ satisfies the gradient flow equation
\begin{equation} \label{eq:2}
\left\{ \begin{aligned}
&dX_t\in -\partial\varphi(X_t)dt+B(t,X_t)dW_t,\ \text{in}\ [0,T]\times E,\\
&X_0=x_0\in \mathcal{F}_e^* \text{~on~} E.
\end{aligned} \right.
\end{equation}

\begin{definition}\label{strong}\textbf{(\cite[Definition 1.3]{GJFA})}
Let $x_0\in L^2(\Omega,\mathcal{F}_0;\mathcal{F}_e^*)$. An $\mathcal{F}_t$-adapted continuous process $X$ in $\mathcal{F}_e^*$ with $X\in L^2(\Omega;C([0,T];\mathcal{F}_e^*))$ and $\partial\varphi(X)\in L^2([0,T]\times\Omega;\mathcal{F}_e^*)$ is called a strong solution to \eqref{eq:2} if $\Bbb{P}$-a.s.
\begin{eqnarray*}
X_t=x_0-\int_0^t\partial\varphi(X_s)ds+\int_0^tB(s,X_s)dW_s,~~\forall t\in[0,T],
\end{eqnarray*}
as an equation in $\mathcal{F}_e^*$.
\end{definition}

\begin{proposition}\label{p1}
If $X$ is a strong solution to \eqref{eq:2} in the sense of Definition \ref{strong}, it is an SVI solution to \eqref{eq:1}.
\end{proposition}
\begin{proof}
Firstly, let us verify regularity. By It\^{o}'s formula,
\begin{eqnarray*}
\Bbb{E}\|X_t\|^2_{\mathcal{F}_e^*}=\Bbb{E}\|x_0\|^2_{\mathcal{F}_e^*}-2\Bbb{E}\int_0^t\langle \partial\varphi(X_s),X_s\rangle_{\mathcal{F}_e^*}ds+\Bbb{E}\int_0^t\|B(s,X_s)\|^2_{L_2(U,\mathcal{F}_e^*)}ds.
\end{eqnarray*}
By the definition of the subdifferential $\partial\varphi$ we have that
$$\langle\partial\varphi(X_s),X_s\rangle_{\mathcal{F}_e^*}=\langle\partial\varphi(X_s),X_s-0\rangle_{\mathcal{F}_e^*}\geq\varphi(X_s),~~ds\otimes \Bbb{P}-a.e.$$
and by \eqref{fe2}
\begin{eqnarray*}
\Bbb{E}\|X_t\|^2_{\mathcal{F}_e^*}+\Bbb{E}\int_0^t\varphi(X_s)ds\leq\Bbb{E}\|x_0\|^2_{\mathcal{F}_e^*}+C_2\Bbb{E}\int_0^t(\|X_s\|^2_{\mathcal{F}_e^*}+1)ds.
\end{eqnarray*}
Hence by Gronwall's lemma we get \eqref{svii}.

To verify the stochastic variational inequality, let $Z\in L^2(\Omega;C([0,T];\mathcal{F}_e^*))\cap L^2([0,T]\times\Omega;L^2(\mu))$ be a solution to
$$Z_t=Z_0+\int_0^tG_sds+\int_0^tB(s,Z_s)dW_s,~~\forall t\in[0,T],$$
for some $\mathcal{F}_t$-progressively measurable process $G\in L^2([0,T]\times \Omega;\mathcal{F}_e^*)$. By It\^{o}'s formula,
\begin{eqnarray*}
\Bbb{E}\|X_t-Z_t\|^2_{\mathcal{F}_e^*}=&&\!\!\!\!\!\!\!\!\Bbb{E}\|x_0-Z_0\|_{\mathcal{F}_e^*}^2+2\Bbb{E}\int_0^t\langle-\partial\varphi(X_s)-G_s,X_s-Z_s\rangle_{\mathcal{F}_e^*}ds\nonumber\\
&&\!\!\!\!\!\!\!\!+\Bbb{E}\int_0^t\|B(s,X_s)-B(s,Z_s)\|^2_{L_2(U,\mathcal{F}_e^*)}ds,~~\forall t\in [0,T].
\end{eqnarray*}
By the definition of the subdifferential $\partial\varphi$ we have that
\begin{eqnarray*}
\langle-\partial\varphi(X_s),X_s-Z_s\rangle_{\mathcal{F}_e^*}\leq \varphi(Z_s)-\varphi(X_s),~~ds\otimes \Bbb{P}-a.e..
\end{eqnarray*}
This together by \eqref{fe1} implies \eqref{svi3}.
\end{proof}

In \cite{RWX2021}, the well-posedness of strong solutions to \eqref{eq:1} was proved if $x_0\in L^2(\mu)\cap L^{2m}(\mu)\cap\mathcal{F}_e^*$ with $m\in[1,\infty)$. Strong solutions to \eqref{eq:1} are defined as follows.
\begin{definition}\label{strong solution}\textbf{(\cite[Definition 3.2]{RWX2021})}
Let $x_0\in\mathcal{F}_e^*$. An $\mathcal{F}_e^*$-valued, $\mathcal{F}_t$-adapted process $X$ is called a strong solution to \eqref{eq:1} if there exists a $\tau\in[2,\infty)$ such that the following conditions hold:
\begin{eqnarray*}
   &&X\ \text{is}\ \mathcal{F}_e^*-{valued\ continuous\ on}\ [0,T],\ \Bbb{P}-a.e.;\\
   &&X\in L^\tau(\Omega \times (0,T)\times E);
\end{eqnarray*}
there is an $\eta \in L^{\frac{\tau}{m}}(\Omega\times (0,T)\times E)$ such that
\begin{eqnarray*}
&&\eta \in \beta(X),\ dt \otimes \Bbb{P} \otimes d\mu-a.e.\ \text{on}\ \Omega \times (0,T)\times E;
\end{eqnarray*}
\begin{eqnarray*}
\int_0^\cdot \eta_sds\in C([0,T];\mathcal{F}_e),\ \Bbb{P}-a.e.;
\end{eqnarray*}
\begin{eqnarray*}
X_t=x+L\int_0^t\eta_sds+\int_0^tB(s,X_s)dW_s,\ \forall t\in [0,T],\ \Bbb{P}-a.e..
\end{eqnarray*}
\end{definition}

\begin{remark}\label{comparison of strong and SVI}
Let $x_0\in L^2(\mu)\cap L^{2m}(\mu)\cap\mathcal{F}_e^*$ ($=L^{2m}(\mu)\cap\mathcal{F}_e^*$, since $\mu$ is a finite measure). Let $X$ be the unique strong solution of \eqref{eq:1} in the sense of Definition \ref{strong solution}. From the proofs of \cite[Theorem 3.3]{RWX2021} and Theorem \ref{theorem} we know that $X$ is also an SVI solution.
\end{remark}

\begin{remark}\label{RWX strong solution}
In \cite{RWX2021}, to obtain the strong solutions to \eqref{eq:1}, it was assumed that the Dirichlet form is given by a squared field operator $\Gamma$, satisfying a weakened chain rule (see \cite[(H4)]{RWX2021}). This assumption is, however, not necessary in this paper.
\end{remark}


\begin{thebibliography}{00}\addtolength{\itemsep}{0.5ex}



\bibitem{AL} G. Alexopoulos, N. Lohoue, \emph{Sobolev inequalities and harmonic functions of polynomial growth}, J. London Math. Soc. (2) 48 (1993) 452-464. \MR{1241781}


\bibitem{BCD} H. Bahouri, J.Y. Chemin, R. Danchin, \emph{Fourier Analysis and Nonlinear Partial Differential Equations}, Grundlehren der mathematischen Wissenschaften [Fundamental Principles of Mathematical Sciences], 343. Springer, Heidelberg, 2011. xvi+523 pp. \MR{2768550}

\bibitem{BGS} L. Ba\u{n}as, B. Gess, M. Neu{\ss}, \emph{Stochastic partial differential equations arising in self-organized criticality}, https://arxiv.org/pdf/2104.13336.pdf.



\bibitem{BTW88} P. Bank, C. Tang, K. Wiesenfeld, \emph{Self-organized criticality}, Phys. Rev. A 38, 364-374 (1988), http://www.chialvo.net/Curso/Cordoba2005/ClasesPowerpoints/Presentacion9/PapersClase9/soc2.pdf.



\bibitem{Barbu} V. Barbu, \emph{Nonlinear Differential Equations of Monotone Type in Banach Spaces}, Springer Monogr. Math., Springer, New York, 2010. \MR{2582280}

\bibitem{BDPTRF} V. Barbu, G. Da Prato, \emph{The two phase stochastic Stefan problem}, Probab. Theory Related Fields 124 (2002), no.4, 544-560. \MR{1942322}

\bibitem{BDRAOP} V. Barbu, G. Da Prato, M. R\"{o}ckner, \emph{Existence of strong solutions for stochastic porous media equation under general monotonicity conditions}. Ann. Probab. 37 (2009), no.2, 428-452.\MR{2510012}

\bibitem{BDRJMAA} V. Barbu, G. Da Prato, M. R\"{o}ckner, \emph{Finite time extinction of solutions to fast diffusion equations driven by linear multiplicative noise}, J. Math. Anal. Appl. 389 (2012), no. 1, 147-164. \MR{2876489}

\bibitem{BDRSIAM} V. Barbu, G. Da Prato, M. R\"{o}ckner, \emph{Stochastic nonlinear diffusion equations with singular diffusivity}, SIAM J. Math. Anal. 41 (2009), no. 3, 1106-1120. \MR{2529957}

\bibitem{BDRSIAM2} V. Barbu, G. Da Prato, M. R\"{o}ckner, \emph{Addendum to: Stochastic nonlinear diffusion equations with singular diffusivity}, BiBoS preprint 12-02-396, http://www.physik.uni-bielefeld.de/bibos/preprints/12-02-396.pdf, 2012, 6 pp.

\bibitem{BDR1} V. Barbu, G. Da Prato, M. R\"{o}ckner, \emph{Stochastic Porous Media Equations}, Springer international Publishing Switzerland, 2016. \MR{3560817}

\bibitem{BDRCMP} V. Barbu, G. Da Prato, M. R\"{o}ckner, \emph{Stochastic porous media equations and self-organized criticality}, Comm. Math. Phys. 285 (2009), no. 3, 901-923. \MR{2470909}

\bibitem{BDRCMP2012} V. Barbu, G. Da Prato, M. R\"{o}ckner, \emph{Stochastic porous media equations and self-organized criticality: convergence to the critical state in all dimensions}, Comm. Math. Phys. 311 (2012), no. 2, 539-555. \MR{2902199}

\bibitem{BRARCH} V. Barbu, M. R\"{o}ckner, \emph{Stochastic variational inequalities and applications to the total variation flow perturbed by linear multiplicative noise}, Arch. Ration. Mech. Anal. 209 (2013), no. 3, 797-834. \MR{3067827}

\bibitem{BRR} V. Barbu, M. R\"{o}ckner, F. Russo, \emph{Stochastic porous media equation in $\mathbb{R}^d$}, J. Math. Pures Appl. (9) 103 (2015), no.4, 1024-1052. \MR{3318178}


\bibitem{BC} H. Bauschke, P.L. Combettes, \emph{Convex Analysis and Monotone Operator Theory in Hilbert Spaces}, Springer, Cham, 2017. xix+619 pp. \MR{3616647}




\bibitem{BensoussanLions} A. Bensoussan, J.L. Lions, \emph{
Applications of Variational Inequalities in Stochastic Control}, Translated from the French
Stud. Math. Appl., 12 North-Holland Publishing Co., Amsterdam-New York, 1982. xi+564 pp. \MR{0653144}


\bibitem{BRSVI} A. Bensoussan, A. Rascanu, \emph{Stochastic variational inequalities in infinite-dimensional spaces}, Numer. Funct. Anal. Optim.18 (1997), no.1-2, 19-54. \MR{1442017}

\bibitem{BH78} J. G. Berryman, C.J. Holland, \emph{Nonlinear diffusion problem arising in plasma physics}, Phys. Rev. Lett. 40 (1978), no. 26, 1720-1722. \MR{0495716}

\bibitem{Bogachev} V.I. Bogachev, \emph{Measure Theory},  Vol. I, II. Springer-Verlag, Berlin, 2007. Vol. I: xviii+500 pp., Vol. II: xiv+575 pp. \MR{2267655}

\bibitem{BBKRSV} K. Bogdan, T. Byczkowski, T. Kulczycki, M. Ryznar, R.M. Song, Z. Vondra\u{c}ek, \emph{Potential Analysis of Stable Processes and Its Extension}, Lecture Notes in Math., vol. 1980, Springer-Verlag, Berlin, 2009. \MR{2569321}

\bibitem{BV} M. Bonforte, J.L. Vazquez, \emph{Quantitative local and global a priori estimates for fractional nonlinear diffusion equations}, (English summary) Adv. Math. 250 (2014), 242-284. \MR{3122168}

\bibitem{book Daprato} G. Da Prato, J. Zabczyk, \emph{Stochastic Equations in Infinite Dimensions}, Second edition
Encyclopedia Math. Appl., 152
Cambridge University Press, Cambridge, 2014. xviii+493 pp. \MR{3236753}

\bibitem{David} A. David, \emph{L\'{e}vy Processes and Stochastic Calculus}, Second edition. Cambridge Studies in Advanced Mathematics, 116. Cambridge University Press, Cambridge, 2009. xxx+460 pp. \MR{2512800}

\bibitem{Davis} E.B. Davies, \emph{Heat Kernels and Spectral Theory}, Cambridge Univ. Press, Cambridge, 1989. \MR{0990239}

\bibitem{Evans} L.C. Evans, \emph{Partial Differential Equations}, Second edition. Graduate Studies in Mathematics, 19. American Mathematical Society, Providence, RI, 2010. xxii+749 pp. \MR{2597943}




\bibitem{FL} I. Fonseca, G. Leoni, \emph{Modern Methods in the Calculus of Variations: $L^p$ Spaces}, Springer Monographs in Mathematics. Springer, New York, 2007. \MR{2341508}

\bibitem{Fukushima} M. Fukushima, \emph{Two topics related to Dirichlet forms: quasi-everywhere convergences and additive functionals}, Dirichlet forms (Varenna, 1992), 21-53, Lecture Notes in Math., 1563, Springer, Berlin, 1993. \MR{1292276}

\bibitem{FOT} M. Fukushima, Y. Oshima, M. Takeda, \emph{Dirichlet Forms and Symmetric Markov Processes}, Second revised and extended edition. Berlin, 2011. \MR{2778606}


\bibitem{GMY} H. Geman, D.B. Madan, M. Yor, \emph{Time changes for L\'{e}vy processes}, Math. Finance 11 (2001), no. 1, 79-96. \MR{1807849}


\bibitem{GCMP} B. Gess, \emph{Finite time extinction for stochastic sign fast diffusion and self-organized criticality}, Comm. Math. Phys. 335 (2015), no. 1, 309-344. \MR{3314506}

\bibitem{GJFA} B. Gess, \emph{Strong solutions for stochastic partial differential equations of gradient type}, J. Funct. Anal. 263 (2012), no. 8, 2355-2383. \MR{2964686}

\bibitem{GRSIAM} B. Gess, M. R\"{o}ckner, \emph{Singular-degenerate multivalued stochastic fast diffusion equations}, SIAM J. Math. Anal. 47 (2015), no. 5, 4058-4090. \MR{3505171}

\bibitem{GRTRAN} B. Gess, M. R\"{o}ckner, \emph{Stochastic variational inequalities and regularity for degenerate stochastic partial differential equations}, Trans. Amer. Math. Soc. 369 (2017) 3017-3045. \MR{3605963}

\bibitem{GTJMPA} B. Gess, J M. T\"{o}lle, \emph{Multi-valued, singular stochastic evolution inclusions}, J. Math. Pures Appl. (9) 101 (2014), no. 6, 789-827. \MR{3205643}

\bibitem{GTJDE} B. Gess, J M. T\"{o}lle,  \emph{Stability of solutions to stochastic partial differential equations}, J. Differential Equations 260 (2016), no. 6, 4973-5025. \MR{3448771}




\bibitem{HaussmannPardoux} U.G. Haussmann, E. Pardoux, \emph{Stochastic variational inequalities of parabolic type}, Appl. Math. Optim, 20 (1989), no.2, 163-192. \MR{0998402}

\bibitem{HKSW} B. B. Hua,  M. Keller, M. Schwarz, M. Wirth, \emph{Sobolev-type inequalities and eigenvalue growth on graphs with finite measure}, Proc. Amer. Math. Soc.151 (2023), no.8, 3401-3414. \MR{4591775}


\bibitem{I} K. It\^{o}, \emph{Lectures on Stochastic Processes}, Tata Institute of Fundamental Research, Bombay, 1960. https://mathweb.tifr.res.in/sites/default/files/publications/ln/tifr24.pdf

\bibitem{Ki88} J.R. King, \emph{Extremely high concentration dopant diffusion in silicon}, IMA Journal of Applied Mathematics, Volume 40, Issue 3, 1988, Pages 163-181, https://doi.org/10.1093/imamat/40.3.163.




\bibitem{LionsStampacchia} J.L. Lions, G. Stampacchia, \emph{Variational inequalities}, Comm. Pure Appl. Math.20 (1967), 493-519. \MR{0216344}


\bibitem{LR} W. Liu, M. R\"{o}ckner, \emph{Stochastic Partial Differential Equations: An Introduction}, Springer International Publishing Switzerland, 2015. \MR{3410409}





\bibitem{MR} Z.M. Ma, M. R\"{o}ckner, \emph{Introduction to the Theory of (Non-Symmetirc) Dirichlet Forms}, Springer-Verlag, Berlin Heidelberg, 1992. \MR{1214375}



\bibitem{Mali} L. Maligranda, \emph{Weakly compact operators and interpolation}, Acta Appl. Math. 27 (1992), no. 1-2, 79-89. \MR{1184880}

\bibitem{Marus} M. Neu{\ss}, \emph{Stochastic partial differential equations arising in self-organized criticality}, Thesis (Ph.D.)-Max Planck Institute in Leipzig, 2020.
\bibitem{N} M. Neu{\ss}, \emph{Well-posedness of SVI solutions to singular-degenerate stochastic porous media equations arising in self-organized criticality}, Stoch. Dyn. 21 (2021), no. 5, Paper No. 2150029, 34 pp. \MR{4297951}

\bibitem{P67} K.R. Parthasarathy, \emph{Probability Measures on Metric Spaces}, Probability and Mathematical Statistics, No. 3 Academic Press, Inc., New York-London 1967 xi+276 pp. \MR{0226684}



\bibitem{Rascanu} A. Rascanu, \emph{Existence for a class of stochastic parabolic variational inequalities}, Stochastics 5 (1981), no.3, 201-239. \MR{0631996}

\bibitem{RRW} J.G. Ren, M. R\"{o}ckner, F.Y. Wang, \emph{Stochastic generalized porous media and fast diffusion equations}, J. Differential Equations 238 (2007), no.1, 118-152. \MR{2334594}

\bibitem{RW} M. R\"{o}ckner, F.Y. Wang, \emph{Non-monotone stochastic generalized porous media equations}, J. Differential Equations 245(2008), no.12, 3898-3935. \MR{2462709}

\bibitem{RWX2021} M. R\"{o}ckner, W.N. Wu, Y.C. Xie, \emph{Stochastic generalized porous media equations over $\sigma$-finite measure spaces with non-continuous diffusivity function}, https://doi.org/10.48550/arXiv.2107.09878.

\bibitem{RWX} M. R\"{o}ckner, W.N. Wu, Y.C. Xie, \emph{Stochastic porous media equation on general measure spaces with increasing Lipschitz nonlinearities},  Stochastic Process. Appl. 128 (2018), no. 6, 2131-2151. \MR{3797655}

\bibitem{RWX2022} M. R\"{o}ckner, W.N. Wu, Y.C. Xie, \emph{Stochastic porous media equation on general measure spaces with increasing Lipschitz nonlinearities}, https://doi.org/10.48550/arXiv.1606.03001.

\bibitem{SSV} R.L. Schilling, R.M. Song, Z. Vondra\u{c}ek, \emph{Bernstein Functions}, Theory and applications. Second edition
De Gruyter Stud. Math., 37
Walter de Gruyter Co., Berlin, 2012. xiv+410 pp. \MR{2978140}


\bibitem{Stampacchia} G. Stampacchia, \emph{Formes bilin\'{e}aires coercitives sur les ensembles convexes}, (French) C. R. Acad. Sci. Paris 258 (1964), 4413-4416. \MR{0166591}



\bibitem{Varo} N. Th. Varopoulos, \emph{Hardy-Littlewood theory for semigroups}, Journal of Functional Analysis 63, 240-260 (1985). \MR{0803094}



\bibitem{VSC}  N. Th. Varopoulos,  L. Saloff-Coste, T. Coulhon, \emph{Analysis and Geometry on Groups}, Cambridge Tracts in Mathematics, 100. Cambridge University Press, Cambridge, 1992. xii+156 pp. \MR{1218884}





\bibitem{V} J.L. Vazquez, \emph{Smoothing and Decay Estimates for Nonlinear Diffusion Equations. Equations of Porous Medium Type}, Oxford Lect. Ser. Math. Appl., vol. 33, Oxford University Press, Oxford, 2006. \MR{2282669}


\bibitem{PMEBOOK} J.L. Vazquez, \emph{The Porous Medium Equation}, Mathematical theory. Oxford Math. Monogr. The Clarendon Press, Oxford University Press, Oxford, 2007. xxii+624 pp. \MR{2286292}



\bibitem{Wang} F.Y. Wang, \emph{Functional Inequalities Markov Semigroups and Spectral Theory}, Science Press, Beijing, 2005, https://www.sciencedirect.com/book/9780080449425/functional-inequalities-markov-semigroups-and-spectral-theory.

\bibitem{WZJDE} W.N. Wu, J.L. Zhai, \emph{Large deviations for stochastic porous media equation on general measure spaces}, J. Differential Equations 269 (2020), no. 11, 10002-10036. \MR{4122641}

\bibitem{WZJEE} W.N. Wu, J.L. Zhai, \emph{Stochastic generalized porous media equations driven by L\'{e}vy noise with increasing Lipschitz nonlinearities}, J. Evol. Equ. 21 (2021), no. 4, 4845-4871. \MR{4350588}

\bibitem{Zhang} Y.C. Zhang, \emph{Scaling theory of self-organized criticality}, Phys. Rev. Lett. 63, 1989. https://journals.aps.org/prl/pdf/10.1103/PhysRevLett.63.470.

\end{thebibliography}
\end{document}